\documentclass[11pt]{amsart}

\usepackage{amsmath}
\usepackage{amsfonts}
\usepackage{amssymb}
\usepackage{latexsym}
\usepackage{graphicx}
\usepackage{enumerate}
\usepackage{multirow}
\usepackage{subfigure}
\usepackage{booktabs}

\usepackage{cite}
\usepackage{hyperref}
\usepackage{url}

\newtheorem{theorem}{Theorem}[section]
\newtheorem{proposition}[theorem]{Proposition}
\newtheorem{lemma}[theorem]{Lemma}
\newtheorem{cor}[theorem]{Corollary}
\newtheorem{remark}{Remark}
\newtheorem{definition}{Definition}
\newtheorem{example}{Example}

\newcommand{\paren}[1]{\left(#1\right)}
\newcommand{\bracket}[1]{\left[#1\right]}
\newcommand{\set}[1]{\left\{#1\right\}}

\newcommand{\ep}{\epsilon}

\newcommand{\be}{\begin{equation}}
\newcommand{\ee}{\end{equation}}
\newcommand{\bes}{\begin{equation*}}
\newcommand{\ees}{\end{equation*}}
\newcommand{\R}{{\bf{R}}}
\newcommand{\T}{{\bf{T}}}

\newcommand{\Z}{{\bf{Z}}}
\newcommand{\N}{{\bf{N}}}
\newcommand{\ds}{\displaystyle}

\newcommand{\xb}{\mathbf{x}}
\newcommand{\kb}{\mathbf{k}}

\numberwithin{equation}{section}

\title[Magma equations in dimension two and higher]{Existence theory for magma equations in dimension two and higher}
\author[Ambrose]{David M. Ambrose}
\address{Drexel University, Department of Mathematics, Philadelphia PA, USA}
\author[Simpson]{Gideon Simpson}
\address{Drexel University, Department of Mathematics, Philadelphia PA, USA}
\author[Wright]{J. Douglas Wright}
\address{Drexel University, Department of Mathematics, Philadelphia PA, USA}
\author[Yang]{Dennis G. Yang}
\address{Drexel University, Department of Mathematics, Philadelphia PA, USA}

\begin{document}

\begin{abstract}

  We examine a degenerate, dispersive, nonlinear wave equation related
  to the evolution of partially molten rock in dimensions two and
  higher.  This simplified model, for a scalar field capturing the
  melt fraction by volume, has been studied by direct numerical
  simulation where it has been observed to develop stable solitary
  waves.  In this work, we prove local in time well-posedness results
  for the time dependent equation, on both the whole space and the
  torus, for dimensions two and higher.  We also prove the existence
  of the solitary wave solutions in dimensions two and higher.

\end{abstract}

\maketitle

\section{Introduction}

Consistent systems of partial differential equations governing the
flow of partially molten rock in the Earth's interior began appearing
in the 1980s in \cite{McKenzie:1984vn,Scott:1986kf}.  These models
captured the essential features of the mechanics of partial melts. Important features include: there are at least two phases (the molten rock, which
behaves as a true fluid, and a residual porous matrix of rock); the
molten rock migrates via porous flow through the residual matrix; on the geologic time scale of tens or hundreds of thousands of years,
the matrix deforms viscously; and inertia can be neglected in both phases due to the highly viscous nature of the problem.

These models typically take the form:
\begin{subequations}
\label{e:magmasystem}
  \begin{gather}
\label{e:melt1}
\partial_t(\rho_{\rm melt} \phi) + \nabla \cdot(\rho_{\rm melt} \phi {\bf v}_{\rm melt}) =\hphantom{-}\text{Melting/Freezing},\\
\label{e:matrix1}
\partial_t(\rho_{\rm matrix} (1-\phi)) + \nabla \cdot(\rho_{\rm matrix} (1-\phi) {\bf v}_{\rm matrix}) = -\text{Melting/Freezing},\\
\label{e:darcy1}
\phi({\bf v}_{\rm melt} - {\bf v}_{\rm matrix}) = - \tfrac{k_{\rm matrix}}{\mu_{\rm melt}}( \nabla p - \rho_{\rm melt} g {\bf z}),\\
\label{e:stress1}
\begin{split}
0 = \bar \rho g {\bf z} - \nabla p &+ \nabla \cdot [2(1-\phi)\mu_{\rm matrix}\dot e_{\rm matrix}]\\
&+ \nabla [(1-\phi)(\zeta_{\rm matrix} - \tfrac{2}{3}\mu_{\rm matrix})\nabla \cdot {\bf v}_{\rm matrix}].
\end{split}
\end{gather}
\end{subequations}

In the above expressions, $\rho_{\rm melt}$ and $\rho_{\rm matrix}$ are the densities of the melt and matrix phases, while ${\bf v}_{\rm melt}$ and ${\bf v}_{\rm matrix}$  are the velocities.  The porosity, $\phi$, is the volume fraction of melt, and $p$ is the pressure.  Thus, \eqref{e:melt1} and \eqref{e:matrix1} merely state conservation of mass between the two phases; additional thermodynamic information must be provided to derive phase changes.  

The reader may recognize \eqref{e:darcy1} as Darcy's Law, with $k_{\rm matrix}$ the permeability of the matrix, $\mu_{\rm melt}$ the shear viscosity of the melt, $p$ the (joint) fluid pressure, and $g{\bf z}$ the buoyancy force.  Lastly, \eqref{e:stress1} reflects the force balance in the matrix, with $\mu_{\rm matrix}$ its shear viscosity, $\dot e_{\rm matrix}$ the strain rate, and $\zeta_{\rm matrix}$ the bulk viscosity.  Constitutive relations must be introduced for the viscosities and the permeability, as they typically depend on the porosity.

Models like this have been revisited, rederived, and extended in a variety of ways, with novel predictions for large scale Earth dynamics; see, for
example,
\cite{Bercovici:2001up,Bercovici:2001vj,Ricard:2001ws,Simpson:2010fo,Simpson:2010fq,Spiegelman:1993vj,Spiegelman:1994wi,
  Katz:2013jx,Takei:2009bz,Takei:2009in,Takei:2009jt,Takei:2013kt,Katz:2006uo}.

\subsection{The Scalar Magma Equation}

Under a series of approximations, one can obtain the reduced model,
\begin{equation}\label{magma}
  \phi_t + \partial_{x_d} \left( \phi^n \right) - \nabla \cdot \left(\phi^n\nabla \phi_t\right) = 0.
\end{equation}
Here, $\phi$ is the {\bf rescaled} porosity, with units of $\phi_0<1$, a characteristic value.

Following the derivation of \cite{Spiegelman:2006fw, simpson2008the},  \eqref{magma} can be obtained from \eqref{e:magmasystem} as follows. First, we neglect phase changes in \eqref{e:melt1} and \eqref{e:matrix1} and assume constant density of each phase.  After dividing out by the density and adding the two equations,
\begin{equation}
\label{e:meanv}
    \nabla \cdot [\phi {\bf v}_{\rm melt} + (1-\phi) {\bf v}_{\rm matrix}]=0.
\end{equation}
The divergence of \eqref{e:darcy1} is then taken, and \eqref{e:meanv} is used to eliminate ${\bf v}_{\rm melt}$,
\begin{equation}
\label{e:darcy2}
    \nabla \cdot {\bf v}_{\rm matrix} = \nabla \cdot \left[\tfrac{k_{\rm matrix}}{\mu_{\rm melt}} (\nabla p - \rho_{\rm melt} g {\bf z})\right].
\end{equation}
If we now assume that the shear viscosities are constants, then, after applying a vector identity,
\begin{equation*}
\begin{split}
    -\mu_{\rm matrix} \nabla \times \nabla \times {\bf v}_{\rm matrix} &+ \nabla \left[ ((1-\phi)\zeta_{\rm matrix} + \tfrac{4}{3}\mu_{\rm matrix})\nabla \cdot {\bf v}_{\rm matrix}\right]\\
    &\quad - (1-\phi) \Delta \rho {\bf z} = \nabla p - \rho_{\rm melt}g {\bf z}.
\end{split}
\end{equation*}
Letting $\mathcal{C} = \nabla \cdot {\bf v}_{\rm matrix}$ be the {\it compaction rate}, and assuming there are no large scale shear motions,
\begin{equation*}
      \nabla \left[ ((1-\phi)\zeta_{\rm matrix} + \tfrac{4}{3}\mu_{\rm matrix})\mathcal{C}\right] - (1-\phi) \Delta \rho {\bf z} = \nabla p - \rho_{\rm melt}g {\bf z}.
\end{equation*}
Multiplying this equation by $k_{\rm matrix}/\mu_{\rm melt}$ and taking the divergence, $\nabla p - \rho_{\rm melt}g {\bf z}$ can be eliminated with \eqref{e:darcy2} to obtain,
\begin{equation}
\label{e:stress2}
    \mathcal{C} - \nabla \cdot \left\{ \tfrac{k_{\rm matrix}}{\mu_{\rm melt}}   \nabla \left[  ((1-\phi)\zeta_{\rm matrix} + \tfrac{4}{3}\mu_{\rm matrix})\mathcal{C}\right]\right \} = - \nabla \cdot \left[\tfrac{k_{\rm matrix}}{\mu_{\rm melt}}  (1-\phi) \Delta \rho g {\bf z}\right].
\end{equation}
If we now make the commonly used approximations that $k_{\rm matrix} =k_0 \phi^n$ and $\zeta_{\rm matrix}$ is constant, then, after nondimensionalization, \eqref{e:stress2} and \eqref{e:matrix1} become
\begin{gather*}
    \partial_{t'} \phi' = - \phi_0 \nabla' \phi' \cdot {\bf v}'_{\rm matrix} + (1- \phi_0 \phi') \mathcal{C}', \\
    \begin{split}
    &\mathcal{C}' - \nabla' \cdot\set{ (\phi')^n \nabla'\bracket{\paren{\frac{(1-\phi_0\phi')\zeta_{\rm matrix} + \tfrac{4}{3}\mu_{\rm matrix})}{(1-\phi_0)\zeta_{\rm matrix} + \tfrac{4}{3}\mu_{\rm matrix})}}\mathcal{C}}} \\
    &\hspace{2.5in} = -\nabla' \cdot \bracket{\frac{1-\phi_0 \phi'}{1-\phi_0} (\phi')^n {\bf z}}.
    \end{split}
\end{gather*}
Finally, if the characteristic porosity is small, i.e., $0< \phi_0\ll 1$, these two equations become
\begin{gather*}
      \partial_{t'} \phi' =  \mathcal{C}', \\
      \mathcal{C}' - \nabla' \cdot\set{ (\phi')^n \nabla'\mathcal{C}'} = -\nabla' \cdot \paren{(\phi')^n {\bf z}}.
\end{gather*}
Substituting the first equation into the second, and dropping the primes, we recover \eqref{magma}.

\subsection{Properties of the Scalar Equation}

In Equation \eqref{magma}, the temporal variable satisfies $t>0$, and the spatial variable satisfies $\xb \in \T^d$ or $\xb \in \R^d$.  The dependent variable $\phi = \phi(\xb,t)\in \R$ is the porosity of the media,
corresponding to volume fraction that is melt.  The nonlinearity,
$\phi^n$, reflects the permeability of the solid rock, with typical
physical values of $n\in[2,3]$.  Equation \eqref{magma} is sometimes
called ``the magma equation,'' a terminology we adopt here, and this
reduced model is the focus of this work.

Some of the key features of \eqref{magma} are that it is dispersive,
degenerate, and nonlocal.  To see the dispersive aspect, if we
formally linearize as
$\phi = \phi_0 + \epsilon e^{i (\kb\cdot \xb - \omega t)}$ with
$\phi_0>0$, one obtains the dispersion relation
\begin{equation*}
  \omega(\kb) = \frac{n \phi_0^{n-1} k_d}{1 + \phi_0^n|\kb|^2}.
\end{equation*}
The degeneracy can be seen in the highest derivative term in the event
that $\phi \to 0$.  The nonlocality can be seen by rewriting
\eqref{magma} as
\begin{subequations}
  \begin{align}
    \label{e:magmasystem1}
    \phi_t & = \mathcal{C},\\
    \label{e:magmasystem2}
    [I - \nabla \cdot(\phi^n \nabla \bullet) ]\mathcal{C} &= - \partial_{x_d}(\phi^n).
  \end{align}
\end{subequations}
Suppose that $\phi$ is given such that $\phi$ is non-negative and sufficiently smooth.  
Then
subject to boundary conditions, \eqref{e:magmasystem2} is elliptic
with respect to $\mathcal{C}$.  Having solved the elliptic equation
for $\mathcal{C}$, \eqref{e:magmasystem1} gives us the time derivative
of $\phi$.  This matter of maintaining ellipticity of
\eqref{e:magmasystem2} is at the heart of analyzing \eqref{magma}.

Numerical simulations of \eqref{magma} have been conducted in
dimensions $d=1,2,3$, usually with the boundary condition that
$\phi\to 1$ as $|\xb|\to \infty$.  This is often approximated on a
large periodic domain.  In addition, initial data tends to break up
into rank-ordered traveling solitary waves.  These waves propagate in
the preferred direction $x_d$ and are radially symmetric in a comoving
frame, and have been observed to be stable.  While they are not
expected to be integrable, the solitary waves have been observed to
have elastic-like collisions.  The solitary waves can be obtained with
the ansatz,
\begin{equation}
  \label{e:travelingwave1}
  \phi(\xb,t) = \bar{Q}(\bar{r}(\xb,t)), \quad \bar{r}^2(\xb,t) = \sum_{j=1}^{d-1} x_j^2 +
  (x_d - \bar{c}t)^2.
\end{equation}
For examples, see
\cite{Simpson:2011fx, Scott:1986kf,scott1984magma,Barcilon:1986wd,Barcilon:1989ve,Wiggins:2012iv}.

While these numerical studies have been fruitful for understanding the
dynamics of the equation, and for gaining insight into the larger
system of equations from which it has been derived, relatively little
has been done to analyze \eqref{magma}.  In \cite{Simpson:2007kx}, the
authors established local and global in time well-posedness results in
$d=1$ by fixed point iteration methods in Sobolev spaces.  In
\cite{Nakayama:1992vi}, the authors proved the existence of traveling
waves in $d=1$ with wave speed $c>n>1$, asymptotically converging to
unity as $x\to \pm \infty$.  This is accomplished in $d=1$ by
obtaining a first integral.  Some additional regularity properties of
the solutions were obtained in \cite{Simpson:2008ia}, and their
stability was studied in that work, along with \cite{Simpson:2008ww}.
However, to the best of the authors' knowledge, no mathematically
rigorous results have been obtained for $d\geq 2$.

\subsection{Main Results}

In this work, we aim to establish local-in-time well-posedness results
for \eqref{magma} in dimensions $d \in \Z$, $d \ge 2$.
We also establish the existence of traveling solitary waves in dimensions $d \in \N$. These
results are:

\begin{theorem}
  \label{t:main1}

  Let $s \in \N$ such that $s> d/2 + \lfloor d/2 \rfloor + 2$ and let
  \[
    U\equiv \left\{ f-1 \in H^{s}(\R^d) : \| 1/f \|_{L^\infty(\R^d)} <
      \infty \text{ and } f > 0 \right\},
  \]
  and assume $\phi_0 \in U$.

  There exists $T_\alpha< 0 < T_\omega$ and
  $\phi \in C^{1}((T_\alpha,T_\omega);U)$ with $\phi(0) = \phi_0$ such
  that $\phi$ satisfies \eqref{magma} for all
  $t \in (T_\alpha,T_\omega)$.

  $T_\omega$ (respectivley $T_\alpha$) is maximal in the following
  sense that we have the dichotomy:
  \begin{enumerate}
  \item $T_\omega = + \infty$ (respectivley $T_\alpha = - \infty$), or
  \item
    $\lim_{t \to T_\omega^-} \left( \| \phi(t)-1 \|_{H^s(\R^d)} +
      \|1/\phi(t)\|_{L^\infty(\R^d)} \right) = + \infty$
    \\(respectivley
    $\ds \lim_{t \to T_\alpha^+} \left( \| \phi(t)-1 \|_{H^s(\R^d)} +
      \|1/\phi(t)\|_{L^\infty(\R^d)} \right) = + \infty$ ).
  \end{enumerate}
\end{theorem}
The essential ingredient in this theorem is that if the data is
continuous and bounded from below away from zero, it
remains so for the time of existence.  While this result is stated in
terms of $\R^d$, an analogous result will be shown to hold in $\T^d$,
where many simulations are performed.

We also prove an existence theorem for radially symmetric traveling solitary
waves satisfying \eqref{e:travelingwave1}.  These solutions are unimodal and converge at spatial infinity
to a constant state. Specifically, we show: 
\begin{theorem}
  \label{t:main2}
  For each $d \in \N$, $n \in [2,3]$, and
  $c \in [1.55, n)$, there
  exists a smooth function $Q : \R_{\ge 0} \to \R$
  with the following properties:
  \begin{enumerate}
  \item $Q(0) = 1$, $Q_r(0) = 0$, and
    $Q_{rr}(0) < 0$. Here, $Q_r$ and
    $Q_{rr}$ denote the first and second derivatives of
    $Q(r)$, respectively.
  \item $Q_r(r) < 0$ for all $r > 0$.
  \item $\lim_{r \to \infty} Q(r)$ exists, and
    $0 < \lim_{r \to \infty} Q(r) < 1$.
  \item For every $\bar{q}_0>0$, putting
    $\phi(\xb,t)=\bar{Q}\left(|\xb - \bar{c}
      t {\bf e}_d |\right) = \bar{q}_0 Q\big( \bar{q}_0^{-\frac{n}
      {2}} |\xb - \bar{c} t {\bf e}_d |\big)$ solves 
      \eqref{magma} with $\bar{c} = \bar{q}_0^{n-1} c$.
  \end{enumerate}
\end{theorem}

\subsection{Outline}
In Section \ref{s:wellposedness}, we prove our local-in-time
well-posedness results.  In Section \ref{s:solitary}, we establish the
existence of the solitary waves.  We conclude with some remarks in
Section \ref{s:discussion}.

\section{Well-posedness}
\label{s:wellposedness}

\subsection{Elliptic Estimates}
We begin by establishing some elliptic estimates for the problem on
$\T^d$ and $\R^d$.  Let
\begin{equation*}
  L_a u:= u - \nabla \cdot (a \nabla u).
\end{equation*}
Throughout, we will assume that $a>0$, either because it is a
continuous function, or because it can be replaced a continuous
version which is positive.  Throughout, we will use $\Pi$ to denote a
polynomial in its arguments, assumed to have positive coefficients.

\subsubsection{Problem on $\T^d$}

First, we recall a standard elliptic regularity result, which can be
found in, for instance, \cite{evans1998pde,gilbarg2001elliptic}.
\begin{theorem}
  \label{evans}
  Suppose that $m \in \N$, $a \in C^{m+1}(\T^d)$, $a>0$,
  $1/a \in L^\infty(\T^d)$, and $g \in H^{m}(\T^d)$.  There exists a
  unique function $u \in H^{m+2}(\T^d)$ such that $L_a u =
  g$. Moreover,
  \begin{equation}\label{evans estimate}
    \| u \|_{H^{m+2}(\T^d)} \le \Pi \left(\|a\|_{C^{m+1}(\T^d)},\|1/a\|_{L^\infty(\T^d)}\right)\|g\|_{H^m(\T^d)}.
  \end{equation}
  $\Pi$ depends only on $m$ and $d$.
\end{theorem}
\label{SET}

Also recall the Sobolev embedding theorem:
\begin{theorem} Let $s > d/2$ and $s \in \N$. There exists
  $C = C(s,d) >0$ such that for all $u \in H^s(\T^d)$ one has
  \begin{equation*}
    \| u \|_{C^{s - \lfloor d/2 \rfloor - 1}(\T^d)} \le C \| u
    \|_{H^s (\T^d)}.
  \end{equation*}
\end{theorem}

This has the following consequence:
\begin{cor}\label{one over}
  Suppose that $s > d$ and $s \in \N$. If $f \in H^s(\T^d)$ and
  $1/f \in L^\infty(\T^d)$, then $1/f \in H^s(\T^d)$ and
$$
\left\| {1 /f} \right\|_{H^s(\T^d)} \le \Pi\left( \|f \|_{H^s(\T^d)},
  \left\| {1/f} \right\|_{L^\infty(\T^d)} \right).
$$
$\Pi$ depends only on $s$ and $d$.
\end{cor}

\begin{proof}
This follows from the previous theorem and the chain
rule; upon differentiating $1/f$ as many as $s$ times,
one arrives at a sum of a number of terms which involve
in the numerator products of up to $s$ derivatives of 
$f,$ with factors of $f$ in the denominator.  The factors
of $f$ in the denominator are bounded in $L^{\infty}$ 
since we have assumed $1/f\in L^{\infty}.$  The products
in the numerator are bounded in $L^{2}$ since in any such product,
only one factor may have more than $s/2$ derivatives, and thus
only one factor might fail to be in $L^{\infty};$ this factor,
then, may be estimated in $L^{2}$ while the remaining factors
may be estimated using Sobolev embedding.
\end{proof}

We next extend \eqref{evans estimate} to obtain an elliptic regularity
result essential to this work,
\begin{proposition}
  \label{boot}
  Let $s > {d / 2} + \lfloor d/2 \rfloor + 1$ and $s \in \N$. Suppose
  that $a \in H^s(\T^d)$, $a>0$, $1/a \in L^\infty(\T^d)$ and
  $g \in H^{s-1}(\T^d)$.  Then there exists a unique function
  $u \in H^{s+1}(\T^d)$ such that $L_a u = g$. Moreover,
  \begin{equation}\label{gen estimate}
    \| u \|_{H^{s+1}(\T^d)} \le \Pi \left(\|a\|_{H^s(\T^d)},\|1/a\|_{L^\infty(\T^d)}\right)\|g\|_{H^{s-1}(\T^d)}.
  \end{equation}
  $\Pi$ depends only on $s$ and $d$.
\end{proposition}

\begin{proof} First, by Theorem \ref{SET} we have
  $a \in C^{s- \lfloor d/2 \rfloor - 1}(\T^d)$. Let
  $m := {s- \lfloor d/2 \rfloor - 2}$. The conditions on $s$ guarantee
  that $0\leq m \leq s-1$.  We therefore find $a \in C^{m+1}(\T^d)$
  and $g \in H^m(\T^d)$. Theorem \ref{evans} implies that there exists
a  unique $u \in H^{m+2}(\T^d)$ for which $L_a u = g$ and the estimate
  in that theorem holds,
 \begin{equation}\label{Hm+2Estimate}
    \|u\|_{H^{m+2}}\leq \Pi(\|a\|_{C^{m+1}},
    \|1/a\|_{L^\infty})\|g\|_{H^m}\leq \Pi(\|a\|_{H^s},
    \|1/a\|_{L^\infty})\|g\|_{H^{s-1}}.
  \end{equation}

  Since $u \in H^2$, we have a strong solution in the sense that we
  can rewrite $L_au = g$ as
$$
-\Delta u =- \frac{1}{a} u + \frac{1}{a} \nabla a \cdot \nabla u + \frac{1}{a} g,
$$
with equality holding in $L^2$. Adding $u$ to both sides gives
\begin{equation}
  \label{e:ell2}
  u - \Delta u = \left(1-\frac{1}{a}\right) u + \frac{1}{a} \nabla a
  \cdot \nabla u + \frac{1}{a} g.
\end{equation}

Next, we observe that if $u \in H^{k+1}$, for any $k >d/2$, then the
right-hand side of \eqref{e:ell2} is in $H^{k\wedge (s-1)}$ (here, we
use the notation $a\wedge b = \min\{a,b\}$);  this follows
from our constraints on $s$ and the assumed regularity of $a$ and
$g$.  Since $u \in H^{m+2},$ we may take $k$ to satisfy $k = m+1 > d/2.$
  By virtue of elliptic regularity, if $u\in H^{k+1}$ and if
the right-hand side of \eqref{e:ell2} is in $H^{k\wedge (s-1)}$, then in fact $u \in H^{\left(k\wedge(s-1)\right)+2}$.
If $k\wedge(s-1)=s-1,$ then this implies $u\in H^{s+1}.$  Otherwise, we may repeat this process.
This bootstrap process can be continued until we conclude $u \in H^{s}$, at
which point the right hand side is limited to being in $H^{s-1},$ because of the regularity of $a$ and $g.$  A final
application of elliptic regularity gives $u \in H^{s+1}$.

To obtain the bound on $\|u\|_{H^{s+1}}$, we take the $H^k$ norm of
both sides of \eqref{e:ell2}, to get
\begin{equation}
  \label{kest} \|u\|_{H^{k+2}(\T^d)} \le \left\| \left( \frac{1}{a}-1\right)
    u\right\|_{H^k(\T^d)} + \left\| \frac{1}{a} \nabla a \cdot \nabla
    u\right\|_{H^k(\T^d)}+ \left \|\frac{1}{a} g\right\|_{H^k(\T^d)}.
\end{equation}
For $k \in (d/2, s-1]\cap \N$,
\begin{equation*}
  \begin{split}
    \|u\|_{H^{k+2}}&\leq C \left( 1+ \left\| 1/a \right\|_{H^k} +
      \left\| 1/a \right\|_{H^k} \| a
      \|_{H^{k+1}}\right)\left\|u\right\|_{H^{k+1}}+ \left\|
      1/a \right\|_{H^k} \left\|g\right\|_{H^k}\\
    & \leq \underbrace{C \left(1+ \left\| 1/a \right\|_{H^s} + \left\|
          1/a \right\|_{H^s} \| a \|_{H^{s}} \right)}_{\equiv
      C_{a}}\|u\|_{H^{k+1}} + \underbrace{\left\| 1/a \right\|_{H^s}
      \left\|g\right\|_{H^{s-1}}}_{\equiv B_{a,g}}.
  \end{split}
\end{equation*}
Note that in this estimate we have used Corollary \ref{one over},
since $s>d$ and thus $1/a\in H^{s}.$  By induction
and our estimate \eqref{Hm+2Estimate} on the $H^{m+2}$ norms, we have:
\begin{equation*}
  \begin{split}
    \|u\|_{H^{s+1}}&\leq C_a^{\lfloor d/2\rfloor }\|u\|_{H^{m+2}} +
    B_{a,g} \sum_{j=0}^{\lfloor d/2\rfloor }C_a^j\\
    &\leq C_a^{\lfloor d/2\rfloor } \Pi(\|a\|_{H^s},
    \|1/a\|_{L^\infty})\|g\|_{H^{s-1}}+ B_{a,g} \sum_{j=0}^{\lfloor
      d/2\rfloor }C_a^j.
  \end{split}
\end{equation*}

The conclusion, \eqref{gen estimate}, now follows by again applying Corollary \ref{one over}.
\end{proof}

\begin{cor}
  \label{lip}
  Let $s > {d / 2} + \lfloor d/2 \rfloor + 1$ and $s \in \N$.  Suppose
  that $a,b \in H^s(\T^d)$, $1/a,1/b \in L^\infty(\T^d)$ and
  $g \in H^{s-1}(\T^d)$.  Let $u\in H^{s+1}(\T^{d})$ and $v \in H^{s+1}(\T^d)$ be the
  functions whose existence is implied by Proposition \ref{boot} such that
  $L_a u =L_b v = g$.  Then
  \begin{equation}
    \label{lip est} \| u -
    v\|_{H^{s+1}} \le \Pi\left(\|
      a\|_{H^{s}},\|1/a\|_{L^\infty},\|
      b\|_{H^{s}},\|1/b\|_{L^\infty}\right)\|g\|_{H^{s-1}}
    \| a - b\|_{H^{s}}.
  \end{equation}
  $\Pi$ depends only on $s$ and $d$.
\end{cor}
\begin{proof}
  Let $\eta := u - v$. A straightforward calculation shows that
  $L_b \eta = \nabla \cdot \left( (a-b) \nabla u \right)=:\tilde{g}$.
  We know that $u \in H^{s+1}(\T^d)$ and $a,b \in H^{s}(\T^d)$ and
  thus $\tilde{g} \in H^{s-1}(\T^d)$. In particular, since the
  condition on $s$ implies that $H^{s}(\T^d)$ is an algebra, we have:
$$
\| \tilde{g} \|_{H^{s-1}} = \| \nabla \cdot \left( (a-b) \nabla u
\right)\|_{H^{s-1}} \le C \| (a-b) \nabla u \|_{H^{s}} \le C\| a- b
\|_{H^s} \| u \|_{H^{s+1}}.
$$
The constant $C>0$ is determined only by $d$.

Applying the estimate from Proposition \ref{boot} shows, therefore,
that
\begin{equation*}
  \begin{split}
    \| \eta \|_{H^{s+1}} &\le
    \Pi\left(\|b\|_{H^s},\|1/b\|_{L^\infty}\right) \|
    \tilde{g}\|_{H^{s-1}} \\
    & \le\Pi\left(\|b\|_{H^s},\|1/b\|_{L^\infty}\right)\| a- b
    \|_{H^s} \| u \|_{H^{s+1}}.
  \end{split}
\end{equation*}
Since we know $L_a u = g$, we apply the estimate from Proposition
\ref{boot} to $\|u\|_{H^{s+1}(\T^d)}$ above. The result is expressed
as \eqref{lip est}.
\end{proof}

Corollary \ref{lip}, of course, directly implies uniqueness
of solutions to the problem $L_{a} u =g.$
We therefore may 
write $L^{-1}_a g$ to mean the unique function $u$ such that
$L_a u = g.$
The results in Proposition \ref{boot} and Corollary \ref{lip} can 
thus be
reformulated as the following pair of estimates which hold when
$s > {d / 2} + \lfloor d/2 \rfloor + 1$:
\begin{equation}\label{boot est} 
\| L^{-1}_a g \|_{H^{s+1}(\T^d)} \le
\Pi\left(\|a\|_{H^s(\T^d)},\|1/a\|_{L^\infty(\T^d)}\right) \|
g\|_{H^{s-1}(\T^d)}
\end{equation}
and
\begin{equation}\label{lip est-2}
  \begin{split}
    &\| [L^{-1}_a - L^{-1}_b] g \|_{H^{s+1}} \\
    &\quad \quad \le \Pi\left(\| a\|_{H^{s}},\|1/a\|_{L^\infty},\|
      b\|_{H^{s}},\|1/b\|_{L^\infty}\right)\|g\|_{H^{s-1}} \| a -
    b\|_{H^{s}}.
  \end{split}
\end{equation}

\subsubsection{Problem on $\R^d$}

Similar elliptic regularity results hold on $\R^{d};$
in \cite{krylov2008lectures}, the following result 
is shown to hold for the
non-divergence form operator $L,$ defined as
\begin{equation*}
  L u = a_{ij}(x) \frac{\partial^2 u}{\partial x_i \partial x_j} +
  b_i(x) \frac{\partial u}{\partial x_i} + c(x) u.
\end{equation*}
\begin{theorem}[See Theorems 9.2.3 and 11.6.2 of \cite{krylov2008lectures}]
  \label{krylov}
  Assume:
  \begin{itemize}
  \item The matrix-valued function $a$ is symmetric, positive
    definite, and  uniformly bounded from above
    and below on $\R^d$.
  \item $a$, $b$, and $c$ are $C^m(\R^d)$, with $C^m$ norms uniformly
    bounded by a constant $K$.
  \item $L 1 \leq -\delta$, for some constant $\delta >0.$  (Here, $L1$ is the operator $L$ applied to the
  constant function $1,$ so $L1=c(x).$)
  \end{itemize}
  Then for any $f \in H^{m}(\R^d)$, there exists a unique
  $u \in H^{m+2}(\R^d)$ such that $Lu=f.$  Furthermore, there is a constant $N$, such
  that for any such $f$, if $Lu=f,$ then
  \begin{equation*}
    \|u\|_{H^{m+2}}\leq N \|f\|_{H^m}.
  \end{equation*}
\end{theorem}
Then the analog of Theorem \ref{evans} holds:
\begin{cor}
  \label{c:Rdevans}
  Assume $a \in C^{m+1}(\R^d)$, $1/a\in L^\infty(\R^d)$ and
  $g\in H^{m}(\R^d)$.  Then there exists a unique $u\in H^{m+2}(\R^d)$
  solving $L_au=g$ with
  \begin{equation}
    \label{e:ellipticRd}
    \| u \|_{H^{m+2}(\R^d)} \le \Pi \left(\|a\|_{C^{m+1}(\R^d)},\|1/a\|_{L^\infty(\R^d)}\right)\|g\|_{H^m(\R^d)}.
  \end{equation}
  $\Pi$ depends only on $m$ and $d$.
\end{cor}

\begin{proof}
  We need to switch signs to agree with the formulation of
  \cite{krylov2008lectures}.  Given $L_a$, define $L$ as
  \begin{equation*}
    Lu = -L_au = a(x) \delta_{ij} \partial^2_{x_i x_j}u +
    a_{x_i} \partial_{x_i}u - u.
  \end{equation*}
  Now, the matrix $a$ of \cite{krylov2008lectures} is equal to $a(x) I$, where $a(x)$ is 
  scalar.  Furthermore, $b_i = a_{x_i}$ and $c(x) = -1$.  Taking $\delta = 1$, we have
  $L1 = -1 = -\delta$ for all $x$.  All of the coefficients are $C^m$,
  so there is a unique solution to $Lu = -g$, hence there is a unique
  solution to $L_a u = g$ in $H^m$.  The bound \eqref{e:ellipticRd}
  can be obtained directly for $m=0$, and by induction for
  $m=1,2,\ldots$.
\end{proof}

We will also make use of the Sobolev embedding theorem:
\begin{theorem} Let $s > d/2$ and $s \in \N$. Then there exists
  $C = C(s,d) >0$ such that for all $u \in H^s(\R^{d})$, one has
  \begin{equation*}
    \| u \|_{C^{s - \lfloor d/2 \rfloor - 1}(\R^d)} \le C \| u
    \|_{H^s (\R^d)}.
  \end{equation*}
\end{theorem}
See Theorem 10 and Remark 12 of Section 10.4 of
\cite{krylov2008lectures} for details.

On $\R^d$, we will assume that our function, $a$, is asymptotic to a
constant as $|x|\rightarrow\infty;$ for convenience, we take this constant to equal one.  This shift in
the far field is the main difference between the $\R^d$ and $\T^d$
problems.  We use the notation $f\in H^{j}(\R^{d})+1$ to indicate $(f-1)\in H^{j}(\R^{d}).$  We have the following result:
\begin{cor}
  Suppose that $s > d$ and $s \in \N$. If $f \in H^s(\R^d) + 1$ and
  $1/f \in L^\infty(\R^d)$, then $1/f \in H^s(\R^d) + 1$ and
  \begin{equation*}
    \left\| {1 /f}-1 \right\|_{H^s} \le \Pi\left( \|f-1 \|_{H^s},
      \left\| {1/f} \right\|_{L^\infty} \right).
  \end{equation*}
  $\Pi$ depends only on $s$ and $d$.
\end{cor}
\begin{proof}
We start by observing that $\frac{1}{f}-1$ is in $L^{2},$ since
$\frac{1}{f}-1=-\frac{f-1}{f},$ and since $f-1\in L^{2}$ and
$\frac{1}{f}\in L^{\infty}.$  Then, upon differentiating
$\frac{1}{f}-1,$ the proof is the same as in the case on the torus.
\end{proof}

The analog of Proposition \ref{boot} holds on $\R^d$:
\begin{proposition}
  \label{boot2} Suppose that $s > {d / 2} + \lfloor d/2 \rfloor + 1$
  and $s \in \N$. Suppose that $a\in H^s(\R^d)+1$,
  $1/a \in L^\infty(\R^d)$ and $g \in H^{s-1}(\R^d)$.  Then there
  exists a unique function $u \in H^{s+1}(\R^d)$ such that
  $L_a u = g$. Moreover,
  \begin{equation*}
    \| u \|_{H^{s+1}(\R^d)} \le \Pi \left(\|a-1\|_{H^s(\R^d)},\|1/a\|_{L^\infty(\R^d)}\right)\|g\|_{H^{s-1}(\R^d)}.
  \end{equation*}
  $\Pi$ depends only on $s$ and $d$.
\end{proposition}

\begin{proof}
  We omit the full details of the proof, as they are almost exactly
  the same as in the case of $\T^d$.  The substantive difference comes
  when we begin to manipulate the analog of \eqref{kest}, in which we estimate the $H^{k+2}$ norm of $u$ by taking the $H^{k}$ norm of both sides of \eqref{e:ell2}.  In 
  particular, in the non-periodic case, norms of $a$ 
  must be treated differently as compared to the periodic
  case, and we estimate as follows:
  \begin{equation*}
    \begin{split}
      \|u\|_{H^{k+2}} &\leq \| (a^{-1}-1)u \|_{H^k} + \|a^{-1} \nabla a
      \cdot
      \nabla u\|_{H^k} + \|a^{-1} g\|_{H^k}\\
      & \leq C_d( \|a^{-1}-1\|_{H^k} + \|a^{-1} -1\|_{H^k} \|a^{-1}
      -1\|_{H^{k+1}} + \|a^{-1}
      -1\|_{H^{k+1}} ) \|u\|_{H^{k+1}} \\
      &\quad + ( 1 +\|a^{-1} -1\|_{H^k} ) \|g\|_{H^{k}}\\
      & \leq C_d( \|a^{-1}-1\|_{H^s} + \|a^{-1} -1\|_{H^s}^2)
      \|u\|_{H^{k+1}} + ( 1 +\|a^{-1} -1\|_{H^s} ) \|g\|_{H^{s-1}}.
    \end{split}
  \end{equation*}

\end{proof}

An analog of Corollary \ref{lip} holds, with the usual shift by
one:
\begin{cor}
  \label{lip2}
  Suppose that $s > {d / 2} + \lfloor d/2 \rfloor + 1$ and $s \in \N$.
  Suppose that $a,b \in H^s(\R^d)+1$, $1/a,1/b \in L^\infty(\R^d)$ and
  $g \in H^{s-1}(\R^d)$.  Let $u$ and $v \in H^{s+1}(\R^d)$ be the
  functions whose existence is implied by Proposition \ref{boot2} and
  for which $L_a u =L_b v = g$.  Then
  \begin{equation*}
    \| u - v\|_{H^{s+1}} \le \Pi\left(\|
      a-1\|_{H^{s}},\|1/a\|_{L^\infty},\|
      b-1\|_{H^{s}},\|1/b\|_{L^\infty}\right)\|g\|_{H^{s-1}}
    \| a - b\|_{H^{s}}.  \end{equation*}
  $\Pi$ depends only on $s$ and $d$.
\end{cor}
We omit the details of this result.

\subsection{Reformulation of \eqref{magma} as an ODE on a Banach Space}

\subsubsection{Problem on $\T^d$}

Suppose that $s > {d / 2} + \lfloor d/2 \rfloor + 1$.  Set
$$
U := \left\{ f \in H^{s}(\T^d) : \| 1/f \|_{L^\infty(\T^d)} < \infty
  \text{ and } f > 0\right\}.
$$
It is straightforward to show that $U$ is an open subset of
$H^s(\T^d)$.  The condition on $s$ implies that $H^s(\T^d)$ is an
algebra and thus $\phi \in U$ implies $\phi^n \in U$ as well.  This in
turn implies, \textit{via} Proposition \ref{boot}, that
$L_{\phi^{n}}^{-1}$ is a well-defined and bounded map from
$H^{s-1}(\T^d)$ into $H^{s+1}(\T^d)$.  Thus we can rewrite
\eqref{magma} as
$$
\phi_t = - L^{-1}_{\phi^n} \left[ \partial_{x_d} \phi^n \right] =:
N(\phi).
$$

We claim that $N$ is a locally Lipschitz map from $U$ into
$H^{s+1}(\T^d)$ (and therefore into $H^s(\T^d)$). First, we use the
triangle inequality:
\begin{multline}
  \label{e:picard1}
    \| N(\phi) - N(\psi) \|_{H^{s+1}(\T^d)} = \| L^{-1}_{\phi^n} \partial_{x_d} \phi^n - L^{-1}_{\psi^n} \partial_{x_d} \psi^n\|_{H^{s+1}(\T^d)}\\
     \le \| L^{-1}_{\phi^n} \partial_{x_d} \phi^n -
    L^{-1}_{\psi^n} \partial_{x_d} \phi^n\|_{H^{s+1}(\T^d)}
    \quad + \| L^{-1}_{\psi^n} \partial_{x_d} \phi^n -
    L^{-1}_{\psi^n} \partial_{x_d} \psi^n\|_{H^{s+1}(\T^d)}.
\end{multline}
Applying inequality \eqref{lip est-2} to the first term on the right-hand side of
\eqref{e:picard1} gives:
\begin{equation*}
  \begin{split}
    &\| L^{-1}_{\phi^n} \partial_{x_d} \phi^n -
    L^{-1}_{\psi^n} \partial_{x_d} \phi^n\|_{H^{s+1}} \\
    &\le \Pi\left( \| \phi^n \|_{H^s},\| \phi^{-n} \|_{L^\infty}, \|
      \psi^n \|_{H^s},\| \psi^{-n} \|_{L^\infty}
    \right)\| \partial_{x_d} \phi^n \|_{H^{s-1}} \| \phi^n -
    \psi^n\|_{H^{s}}.
  \end{split}
\end{equation*}
Since
$\| \partial_{x_d} \phi^n \|_{H^{s-1}(\T^d)} \le C \| \phi
\|_{H^{s}(\T^d)}^n$, we have:
\begin{equation*}
    \| L^{-1}_{\phi^n} \partial_{x_d} \phi^n - L^{-1}_{\psi^n} \partial_{x_d} \phi^n\|_{H^{s+1}}
    \le \Pi\left( \| \phi \|_{H^s}, \|1/\phi\|_{L^\infty},\| \psi
      \|_{H^s},\|1/\psi\|_{L^\infty} \right) \| \phi^n -
    \psi^n\|_{H^{s}}.
\end{equation*}
(Recall that $\Pi$ denotes a polynomial in its arguments with positive
coefficients. Here, it is wholly determined by $s$, $d$ and $n$.)
Factoring, and using that $H^s(\T^d)$ is an algebra, gives
$\| \phi^n - \psi^n\|_{H^{s}(\T^d)} \le
\Pi(\|\phi\|_{H^s(\T^d)},\|\psi\|_{H^s(\T^d)})\| \phi - \psi
\|_{H^{s}(\T^d)}.$ Thus, all together we have:
\begin{multline*}
  \| L^{-1}_{\phi^n} \partial_{x_d} \phi^n -
  L^{-1}_{\psi^n} \partial_{x_d} \phi^n\|_{H^{s+1}(\T^d)} \\ \le
  \Pi\left( \| \phi \|_{H^s(\T^d)}, \|1/\phi\|_{L^\infty(\T^d)},\|
    \psi \|_{H^s(\T^d)},\|1/\psi\|_{L^\infty(\T^d)} \right) \| \phi -
  \psi\|_{H^{s}(\T^d)}.
\end{multline*}

The second term on the right-hand side in estimate \eqref{e:picard1} can be
estimated using \eqref{boot est} as follows:
$$
\| L^{-1}_{\psi^n} \partial_{x_d} \phi^n -
L^{-1}_{\psi^n} \partial_{x_d} \psi^n\|_{H^{s+1}(\T^d)} \le \Pi\left(
  \| \psi^n \|_{H^s(\T^d)}, \| \psi^{-n} \|_{L^\infty(\T^d)}
\right)\| \partial_{x_d} (\phi^n - \psi^n )\|_{H^{s-1}(\T^d)}.
$$
Arguments exactly like those used above show that
\begin{multline*}
  \| L^{-1}_{\psi^n} \partial_{x_d} \phi^n -
  L^{-1}_{\psi^n} \partial_{x_d} \psi^n\|_{H^{s+1}(\T^d)} \\ \le
  \Pi\left( \| \phi \|_{H^s(\T^d)}, \|1/\phi\|_{L^\infty(\T^d)},\|
    \psi \|_{H^s(\T^d)},\|1/\psi\|_{L^\infty(\T^d)} \right) \| \phi -
  \psi\|_{H^{s}(\T^d)}.\end{multline*} And so all together, if $\phi$
and $\psi \in U$ then:
$$
\| N(\phi) - N(\psi) \|_{H^{s+1}(\T^d)} \le \Pi\left( \| \phi
  \|_{H^s(\T^d)}, \|1/\phi\|_{L^\infty(\T^d)},\| \psi
  \|_{H^s(\T^d)},\|1/\psi\|_{L^\infty(\T^d)} \right) \| \phi -
\psi\|_{H^{s}(\T^d)}.
$$
Likewise one has, for $\phi \in U$:
$$
\|N(\phi) \|_{H^{s+1}(\T^d)} \le \Pi( \| \phi
\|_{H^s(\T^d)},\|1/\phi\|_{L^\infty(\T^d)}) \| \phi \|_{H^{s}(\T^d)}.
$$
These last two estimates imply that $N$ is a locally Lipschitz map on
$U$ into $H^{s+1}(\T^d)\subset H^{s}(\T^d)$.  Thus we have, using the
Picard Theorem \cite{zeidler}:

\begin{theorem} Let $s \in \N$ such that
  $s> d/2 + \lfloor d/2 \rfloor + 2$ and let
  $$ U:= \left\{ f \in H^{s}(\T^d) : \| 1/f \|_{L^\infty(\T^d)} <
    \infty \text{ and } f > 0\right\}.$$ Suppose that $\phi_0 \in
  U$. Then there exists $T_\alpha< 0 < T_\omega$ and
  $\phi \in C^{1}((T_\alpha,T_\omega);U)$ with $\phi(0) = \phi_0$ such
  that $\phi$ satisfies \eqref{magma} for all
  $t \in (T_\alpha,T_\omega)$. The value $T_\omega$ (resp. $T_\alpha$) is
  maximal in the following sense that we have the dichotomy:
  \begin{enumerate}
  \item $T_\omega = + \infty$ (resp. $T_\alpha = - \infty$), or
  \item
    $\lim_{t \to T_\omega^-} \left( \| \phi(t) \|_{H^s(\T^d)} +
      \|1/\phi(t)\|_{L^\infty(\T^d)} \right) = + \infty$ \\(resp.
    $\ds \lim_{t \to T_\alpha^+} \left( \| \phi(t) \|_{H^s(\T^d)} +
      \|1/\phi(t)\|_{L^\infty(\T^d)} \right) = + \infty$ ).
  \end{enumerate}
  The function $\phi$ is the unique function with the properties
  listed above.  Lastly, the map from $\phi_0$ to the solution $\phi$
  is continuous. That is to say, fix $\phi_0 \in U$ and let $\phi(t)$
  be the solution described above. Then for all $\ep > 0$ and compact
  sets $I \subset (T_\alpha,T_\omega)$ there exists $\delta > 0$ such
  that $v_{0}\in U$ and $\| v_0 - u_0 \|_{H^{s}(\T^d)} \le \delta$ implies
  $\sup_{t \in I} \| u(t) - v(t) \| \le \epsilon$. Here,
  $v \in C^1(I;U)$ solves \eqref{magma} with $v(0) = v_0$.
\end{theorem}

\subsubsection{Problem on $\R^d$}

The argument for the $\T^d$ case adapts to $\R^d$, with the shift of
the function at infinity, giving us:
\begin{theorem} Let $s \in \N$ such that
  $s> d/2 + \lfloor d/2 \rfloor + 2$ and let
  $$ U:= \left\{ f-1 \in H^{s}(\R^d) : \| 1/f \|_{L^\infty(\T^d)} <
    \infty \text{ and } f > 0\right\}.$$ Suppose that $\phi_0 \in
  U$. Then there exists $T_\alpha< 0 < T_\omega$ and
  $\phi \in C^{1}((T_\alpha,T_\omega);U)$ with $\phi(0) = \phi_0$ such
  that $\phi$ satisfies \eqref{magma} for all
  $t \in (T_\alpha,T_\omega)$. The value $T_\omega$ (resp. $T_\alpha$) is
  maximal in the following sense that we have the dichotomy:
  \begin{enumerate}
  \item $T_\omega = + \infty$ (resp. $T_\alpha = - \infty$), or
  \item
    $\lim_{t \to T_\omega^-} \left( \| \phi(t)-1 \|_{H^s(\R^d)} +
      \|1/\phi(t)\|_{L^\infty(\R^d)} \right) = + \infty$ \\(resp.
    $\ds \lim_{t \to T_\alpha^+} \left( \| \phi(t)-1 \|_{H^s(\R^d)} +
      \|1/\phi(t)\|_{L^\infty(\R^d)} \right) = + \infty$ ).
  \end{enumerate}
  The function $\phi$ is the unique function with the properties
  listed above.  Lastly, the map from $\phi_0$ to the solution $\phi$
  is continuous. That is to say, fix $\phi_0 \in U$ and let $\phi(t)$
  be the solution described above. Then for all $\ep > 0$ and compact
  sets $I \subset (T_\alpha,T_\omega)$ there exists $\delta > 0$ such
  that $v_{0}\in U$ and $\| v_0 - u_0 \|_{H^{s}(\R^d)} \le \delta$ implies
  $\sup_{t \in I} \| u(t) - v(t) \| \le \epsilon$. Here,
  $v \in C^1(I;U)$ solves \eqref{magma} with $v(0) = v_0$.
\end{theorem}

\section{Solitary Waves}
\label{s:solitary}

As noted in the introduction, solitary wave solutions are
readily observed in numerical simulations of \eqref{magma} (see, for instance, \cite{Simpson:2011fx, Scott:1986kf,scott1984magma,Barcilon:1986wd,Barcilon:1989ve,Wiggins:2012iv}), and can be
formally obtained from the {\it ansatz} \eqref{e:travelingwave1},
$\phi(\xb,t) := \bar{Q}(\left \vert \xb - \bar{c}t \mathbf{e}_d
\right\vert )$.  The profile $\bar{Q}=\bar{Q}(\bar{r})$ must then
solve:
\begin{subequations} \label{IVP0}
  \begin{equation} \label{ODE0}
  -\bar{c}\bar{Q}_{\bar{r}} + (\bar{Q}^n)_{\bar{r}} + \bar{c} (
  \bar{Q}^n \bar{Q}_{\bar{r}\bar{r}})_{\bar{r}} + \bar{c}(d-1)
  \bar{Q}^n \Big(\frac{1}{\bar{r}} \bar{Q}_{\bar{r}}\Big)_{\bar{r}} =
  0.
  \end{equation} 
The goal is to find speeds $\bar{c}>0$ and bounded, positive, non-constant
solutions $\bar{Q}(\bar{r}) \in C^3(\R_{\ge 0})$ of \eqref{ODE0} satisfying
  \begin{equation} \label{IC0} 
  \bar{Q}(0) = \bar{q}_0 > 0, \quad
  \bar{Q}_{\bar{r}}(0) = 0, \quad \bar{Q}_{\bar{r}\bar{r}}(0) =
  \bar{\mu} < 0 ,
  \end{equation}
\end{subequations}
where $\bar{q}_0$ and $\bar{\mu}$ are also free parameters just like
the speed $\bar{c}$.

When $d = 1$, the final term in the left-hand side of \eqref{ODE0}
vanishes, and the resulting ODE is simple enough to allow for
classical phase plane analysis and stable manifold theorem analysis,
yielding existence of solutions and decay
properties \cite{Simpson:2008ia, Nakayama:1992vi}.  When $d > 1$, the
situation is more complicated because the ODE \eqref{ODE0} is
non-autonomous and not integrable. To reduce the number of free
parameters, we rescale the variables as follows:
\begin{equation} \label{res}
\bar{Q} = \bar{q}_0 Q , \quad \bar{r} = \bar{q}_0^{\frac{n}{2}} r ,
\quad \bar{c} = \bar{q}_0^{n-1} c , \quad \bar{\mu} = \bar{q}_0^{1-n}
\mu .
\end{equation}
This turns \eqref{IVP0} into the following initial value problem (IVP):
\begin{subequations} \label{IVP}
  \begin{gather}
    -Q_r + \frac{1}{c}(Q^n)_r + ( Q^n Q_{rr} )_r + (d-1)Q^n \Big(\frac{1}{r} Q_r \Big)_r = 0 , \label{ODE} \\
    Q(0) = 1, \quad Q_r(0) = 0, \quad Q_{rr}(0) = \mu < 0 . \label{IC}
  \end{gather}
\end{subequations}

Note that the occurrence of the $1/r$ singularity in the ODE
\eqref{ODE} prohibits the application of the standard theory of
existence, uniqueness, and continuous dependence for regular initial
value problems. However, the term $Q_r/r$ is ``not really singular''
under the initial condition $Q_r(0) = 0$ since it implies that
$\lim_{r \rightarrow 0^+} Q_r/r = Q_{rr}(0) = \mu$. Consequently, we
still have existence, uniqueness, and continuous dependence for the
IVP \eqref{IVP} as stated in the following lemma.

\begin{lemma} \label{lem EUC} The following are true for the IVP
  \eqref{IVP}:
  \begin{enumerate}
  \item For every $d$, $n$, $c\, (\neq 0)$, and $\mu$, there exists 
    $\sigma > 0$ such that the IVP \eqref{IVP} has a unique solution
    $Q(r) \in C^3([0, \sigma])$.
  \item If for $d$, $n$, $c\, (\neq 0)$, and $\mu$, the IVP
    \eqref{IVP} has a solution $Q(r) \in C^3([0, \sigma])$ for some
    $\sigma > 0$, then for every $\Delta > 0$, there exists a
    $\delta > 0$ such that for every
    $d' \in (d - \delta, d + \delta)$,
    $n' \in (n - \delta, n + \delta)$,
    $c' \in (c - \delta, c + \delta)$, and
    $\mu' \in (\mu - \delta, \mu + \delta)$, the IVP \eqref{IVP} also
    has a solution $Q'(r) \in C^3([0, \sigma])$ and
    $\|Q' - Q\|_{C^3([0, \sigma])} \le \Delta$ with
    $\|\cdot\|_{C^3([0, \sigma])}$ being the standard $C^3$-norm on
    $C^3([0, \sigma])$.
  \end{enumerate}
\end{lemma}

The proof of Lemma \ref{lem EUC} is just a modification of the
standard proof, and it makes use of the fact that $Q_r \sim \mu r$
when $r \rightarrow 0^+$. We omit the proof here.

With the existence, uniqueness, and continuous dependence for the IVP
\eqref{IVP} established, the next proposition asserts the existence of strictly
decreasing, positive solutions $Q(r) \in C^3(\R_{\ge 0})$. By virtue
of the rescaling \eqref{res}, each of these solutions gives rise to a
one-parameter family of traveling wave solutions to \eqref{magma} as stated in Theorem \ref{t:main2}.

\begin{proposition} \label{prp Q} Suppose that $d \in \N$ and
  $n \in [2, 3]$. Then for each $c \in [1.55, n)$, there exists
  $\mu_c < 0$ such that the IVP \eqref{IVP} with $\mu = \mu_c$ has a
  solution $Q(r) \in C^3(\R_{\ge 0})$ satisfying that $Q_r(r) < 0$ for
  all $r > 0$ and $0 < \lim_{r \rightarrow \infty} Q(r) < 1$.
\end{proposition}

Furthermore, we prove the following sufficient condition for the exponential convergence of the monotonic solution $Q(r)$:

\begin{proposition}\label{prp suff cond}
  If $Q_r(r) < 0$ for all $r \in (0,\infty)$ and
  $0 < \lim_{r \rightarrow \infty} Q(r) =: Q_{\tau} <
  (\frac{c}{n})^{\frac{1}{n-1}}$, then there exist $M,k > 0$ such
  that $|Q(r) - Q_{\tau}| \le M e^{-kr}$ for all $r > 0$.
\end{proposition}

We prove Proposition \ref{prp Q} in Section \ref{prf prp Q}, and we prove Proposition \ref{prp suff cond} in Section \ref{prf prp exp conv}. In addition, we present an example at the end of Section \ref{prf prp exp conv} with numerical evidence suggesting that the conditions of Proposition \ref{prp suff cond} hold for some specific values of $d$, $n$, and $c$ and hence $Q(r)$ converges exponentially to $Q_{\tau}$ for those parameter values. Our strategy is to establish a lower bound on $\mu_c$ by numerically integrating (\ref{IVP}) with certain $\mu$. Then, together with some constructions we use in the proof of Proposition \ref{prp Q}, this lower bound on $\mu_c$ implies that $Q_{\tau} < (\frac{c}{n})^{\frac{1}{n-1}}$. Although we do not have a proof of the general exponential convergence of $Q(r)$, it is our belief that the conditions of Proposition \ref{prp suff cond} hold for most (if not all) $d$, $n$, and $c$ as specified in Proposition \ref{prp Q}.

\subsection{Proof of Proposition \ref{prp Q}} \label{prf prp Q}

In view of item (1) of Lemma \ref{lem EUC} and the initial conditions
\eqref{IC}, for each solution $Q(r)$ of the IVP \eqref{IVP}, we define
\begin{equation} \label{tau}
\tau := \sup \big\{ \sigma > 0 \;\big|\; \text{$Q(r) \in C^3([0, \sigma])$, and $Q(r) > 0$ and $Q_r(r) < 0$ for all $r \in (0, \sigma]$} \big\} .
\end{equation}
Clearly, $\tau$ depends on
the parameters $c$ and $\mu$ for the IVP \eqref{IVP}, and it is
possible that $\tau = \infty$. By the definition of $\tau$, we have
that $Q(r)$ is of class $C^3$ on $[0, \tau)$ but its derivatives may
not be bounded as $r \rightarrow \tau^-$ and that $Q(r) > 0$ and
$Q_r(r) < 0$ for all $r \in (0,\tau)$. Thus, regardless of whether
$\tau$ is finite or not,
$Q_{\tau} := \lim_{r \rightarrow \tau^-} Q(r)$ always exists, and
$0 \le Q_{\tau} < 1$. In what follows, we will first construct some
technical estimates valid for $r \in (0,\tau)$. Then we will prove the
proposition using a shooting argument based on these estimates.

Since $Q(r)$ is strictly decreasing on $[0, \tau)$, it has an inverse
$r(Q)$ defined for all $Q \in (Q_{\tau}, 1]$. Then, for any function
$f(r)$, we have that for all $r \in [0, \tau)$ and correspondingly
$Q = Q(r) \in (Q_{\tau}, 1]$,
\begin{equation} \label{r to q}
\int_0^r f Q_r dr = \int_1^Q f dq ,
\end{equation}
where $dq = Q_r dr$ and `$f$' represents $f(r)$ in
the first integral and $f(r(q))$ in the second integral. In the
subsequent derivations, our notation will follow this convention
whenever we substitute the variable of integration in the form of
\eqref{r to q}.

Dividing \eqref{ODE} by $Q^n$ and then integrating from $0$ to $r$
gives 
\begin{equation*}
\frac{1}{n-1} \left(\frac{1}{Q^{n-1}} - 1\right) +
\frac{n}{c}\ln{Q} + (Q_{rr} - \mu) + n \int_0^r \frac{Q_{rr} Q_r}{Q}
dr + (d-1) \left( \frac{Q_r}{r} - \mu \right) =0 .
\end{equation*}
Integrating the remaining integral by parts and using \eqref{r to q}, we have that $\int_0^r \frac{Q_{rr} Q_r}{Q} dr = \frac{Q_r^2}{2 Q} + \frac{1}{2} \int_1^Q \frac{Q_r^2}{q^2} dq$. Then after rearranging terms in the above equation, we obtain that
\begin{equation} \label{Qrr}
Q_{rr} + \frac{n Q_r^2}{2 Q} = F_1(Q,\mu) - \frac{n}{2} \int_1^Q \frac{Q_r^2}{q^2} dq - (d-1) \frac{Q_r}{r} ,
\end{equation}
where 
\begin{equation} \label{F1}
F_1(Q,\mu) := - \frac{1}{n-1}
\left(\frac{1}{Q^{n-1}} - 1\right) - \frac{n}{c}\ln{Q} + \mu d .
\end{equation}
Next, multiplying both sides of \eqref{Qrr} by $Q^n Q_r$ and then
integrating again from $0$ to $r$ leads to
\begin{equation} \label{Qr2}
\frac{1}{2}Q^n Q_r^2 = \int_1^Q F_1(q,\mu) q^n dq - \int_1^Q \left(
  \frac{n}{2} \int_1^q \frac{Q_r^2}{\hat{q}^2} d\hat{q} + (d-1)
  \frac{Q_r}{r} \right) q^n dq .
\end{equation}
For all $r \in (0, \tau)$ and correspondingly $Q = Q(r) \in (Q_{\tau}, 1)$, the second integral in the right-hand side is strictly positive. It follows that
\begin{equation} \label{bound}
Q_r^2 < 2 F_2(Q,\mu) Q^{-n} ,
\end{equation}
where
\begin{equation} \label{F2}
F_2(Q,\mu) := \int_1^Q F_1(q,\mu) q^n dq .
\end{equation}
Finally, we define 
\begin{equation} \label{F3}
F_3(Q,\mu) := F_1(Q,\mu) - n
\int_1^Q F_2(q,\mu) q^{-(n+2)} dq .
\end{equation}
Then, in view of \eqref{Qrr}
and \eqref{bound}, we have that
\begin{equation} \label{DE bound}
Q_{rr} + \frac{n Q_r^2}{2 Q} + (d-1) \frac{Q_r}{r} < F_3(Q,\mu) .
\end{equation}

We emphasize that although inequalities \eqref{bound} and \eqref{DE bound} require $\mu < 0$ and are restricted to $r \in (0, \tau)$ and
correspondingly $Q = Q(r) \in (Q_{\tau}, 1)$, the functions
$F_i(Q,\mu)$, $i = 1, 2, 3$, as given by \eqref{F1}, \eqref{F2}, and
\eqref{F3}, are defined for all $(Q,\mu) \in \R_{>0} \times
\R$. Furthermore, these functions can be put in the form of 
\begin{equation} \label{Fi}
F_i(Q,\mu) = g_i(Q) + h_i(Q) \mu .
\end{equation}
The expressions of
$g_1$ and $h_1$ can be extracted from \eqref{F1} immediately. One can
also obtain the closed-form expressions of $g_2$, $h_2$, $g_3$, and
$h_3$ after solving the integrals in \eqref{F2} and \eqref{F3}. In
particular, we have
\begin{align}
  h_1(Q) \equiv{}& d , \notag \\
  h_2(Q) ={}& \frac{d}{n+1} (Q^{n+1} - 1) , \notag \\
  h_3(Q) ={}& d \left( 1 - \frac{n \ln{Q}}{n+1} - \frac{n}{(n+1)^2} \left(\frac{1}{Q^{n+1}}-1\right) \right). \label{h3}
\end{align}
The closed-form expressions of $g_1$, $g_2$, and $g_3$ are omitted here. Note that $h_1$ is always positive and that $h_2(Q)<0$ for $Q < 1$ and $h_2(Q)=0$ only when $Q=1$. In addition, it is easy to check that for each $n > 0$, there is a unique $Q_* \in (0,1)$, depending on $n$ only, such that $h_3(Q_*) = 0$. Furthermore, $h_3(Q) > 0$ for $Q \in (Q_*, 1]$. Then $h_1$, $h_2$, and $h_3$ are all nonzero for any $Q \in (Q_*, 1)$. For $i=1,2,3$, define $\mu_i : (Q_*, 1) \rightarrow \R$ as follows:
\begin{equation*}
\mu_i(Q) := -\frac{g_i(Q)}{h_i(Q)} .
\end{equation*}
In view of \eqref{Fi}, we
have that for any $Q \in (Q_*, 1)$, $F_i(Q,\mu) = 0$ only when
$\mu = \mu_i(Q)$, $i=1,2,3$. Furthermore, by the signs of $h_i$ on the interval $(Q_*, 1)$, we can determine the signs of $F_i(Q,\mu)$ for $(Q,\mu) \in (Q_*, 1) \times \R$ as shown in Figure \ref{FIG Fi}.

\begin{figure}[h]
  \centering 
    \subfigure[The graph of $\mu_1(Q)$. For $Q_* < Q < 1$,
    $F_1(Q,\mu)$ is positive in the region above the graph of 
    $\mu_1(Q)$ and is negative in the region below the graph
    of $\mu_1(Q)$.]{\label{FIG F1}\includegraphics[scale=0.35]{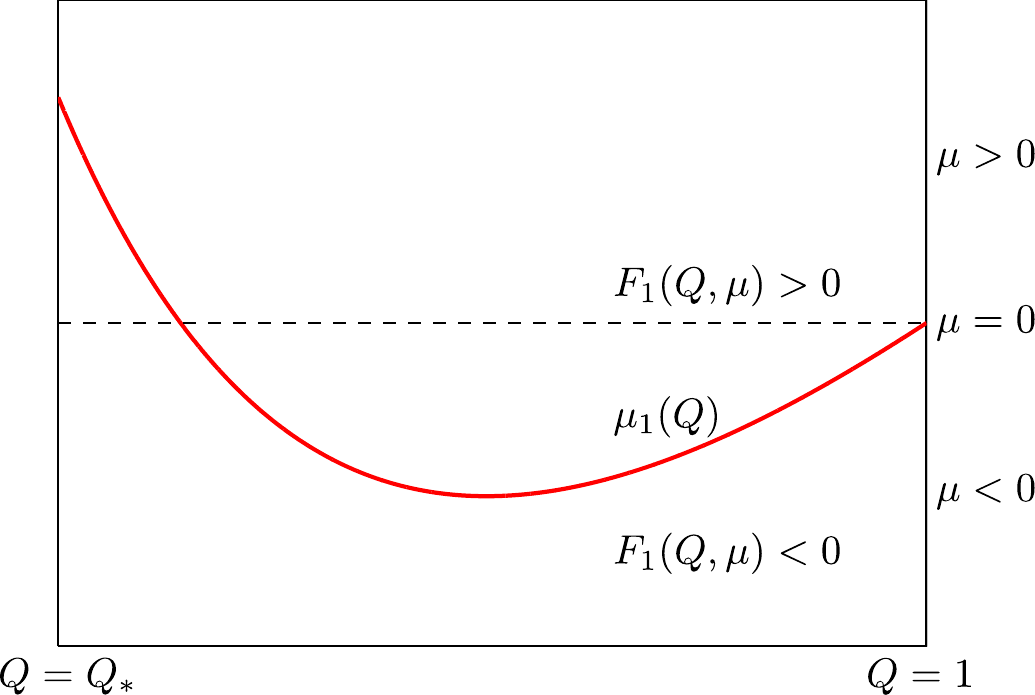}} \hspace{0.05in}
    \subfigure[The graph of $\mu_2(Q)$. For $Q_* < Q < 1$,
    $F_2(Q,\mu)$ is negative in the region above the graph of 
    $\mu_2(Q)$ and is positive in the region below the graph
    of $\mu_2(Q)$.]{\label{FIG F2}\includegraphics[scale=0.35]{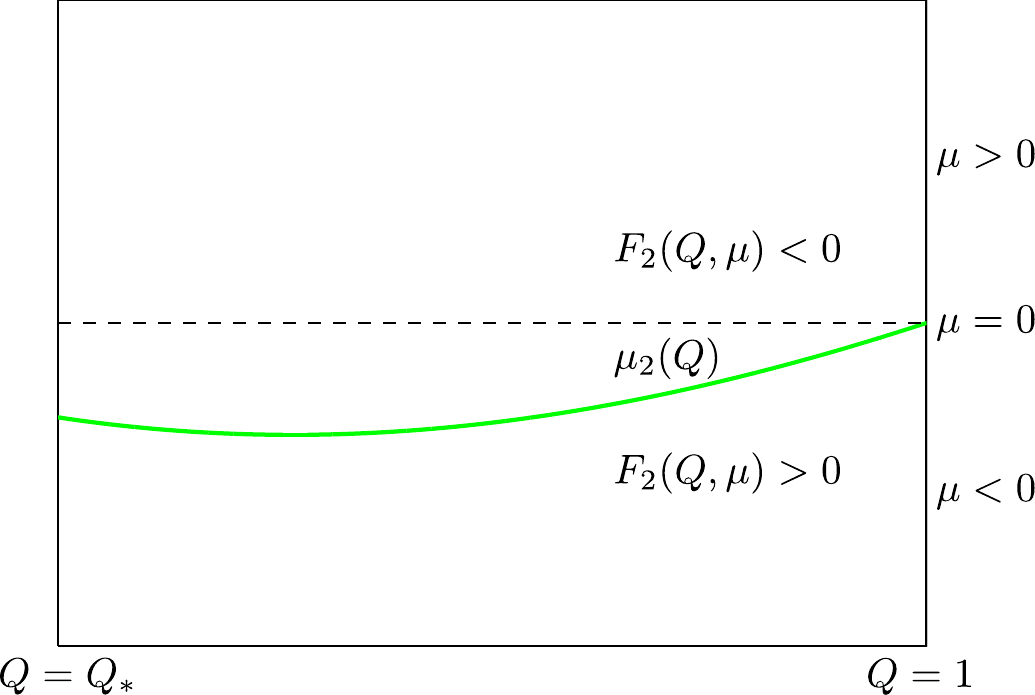}} \\
    \subfigure[The graph of $\mu_3(Q)$. For $Q_* < Q < 1$,
    $F_3(Q,\mu)$ is positive in the region above the graph of 
    $\mu_3(Q)$ and is negative in the region below the graph
    of $\mu_3(Q)$.]{\label{FIG F3}\includegraphics[scale=0.35]{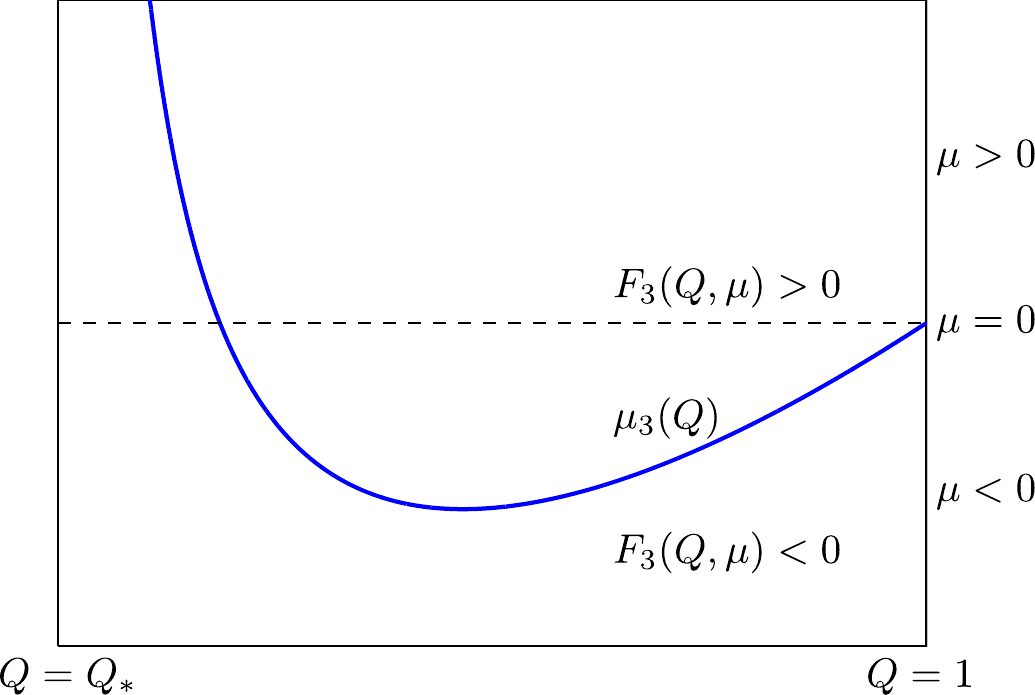}}
  \caption{The graphs of $\mu_i(Q)$ and the signs of $F_i(Q,\mu)$.}
  \label{FIG Fi}
\end{figure}

Furthermore, the graphs of $\mu_i(Q)$, as shown in Figure \ref{FIG MUi}, have the following properties.

\begin{lemma} \label{LEM mui properties} Suppose that $d > 0$, $n \in [2, 3]$, and $c \in [1.55, n)$. Then the following are true:
  \begin{enumerate}
  \item For $i = 1, 2, 3$, $\mu_i(Q) \rightarrow 0^-$ as
    $Q \rightarrow 1^-$.
  \item For each $i = 1, 2, 3$, $\mu_i(Q)$ has a unique global minimum
    at $Q = Q_i \in (Q_*, 1)$.
    \begin{enumerate}
      \renewcommand{\labelenumii}{(2\alph{enumii})}
    \item $Q_1 = (\frac{c}{n})^{\frac{1}{n-1}}$.
    \item $Q_2 < Q_1$, and the graphs of $\mu_1(Q)$ and $\mu_2(Q)$
      intersect at $Q = Q_2$.
    \item $\mu_3(Q_3) < \mu_1(Q_1) < \mu_2(Q_2) < 0$. \label{mui
        order}
    \end{enumerate}
  \item $F_3(Q_*,\mu) = g_3(Q_*) < 0$ for all $\mu \in \R$, and
    $\mu_3(Q) \rightarrow \infty$ as $Q \rightarrow Q_*^+$.
  \item $\mu_3(Q) < \mu_1(Q)$ for $Q \in (Q_2, 1)$.
  \end{enumerate}
\end{lemma}

\begin{figure}[h]
  \centering
  \hspace{0.55in}\includegraphics[scale=0.5]{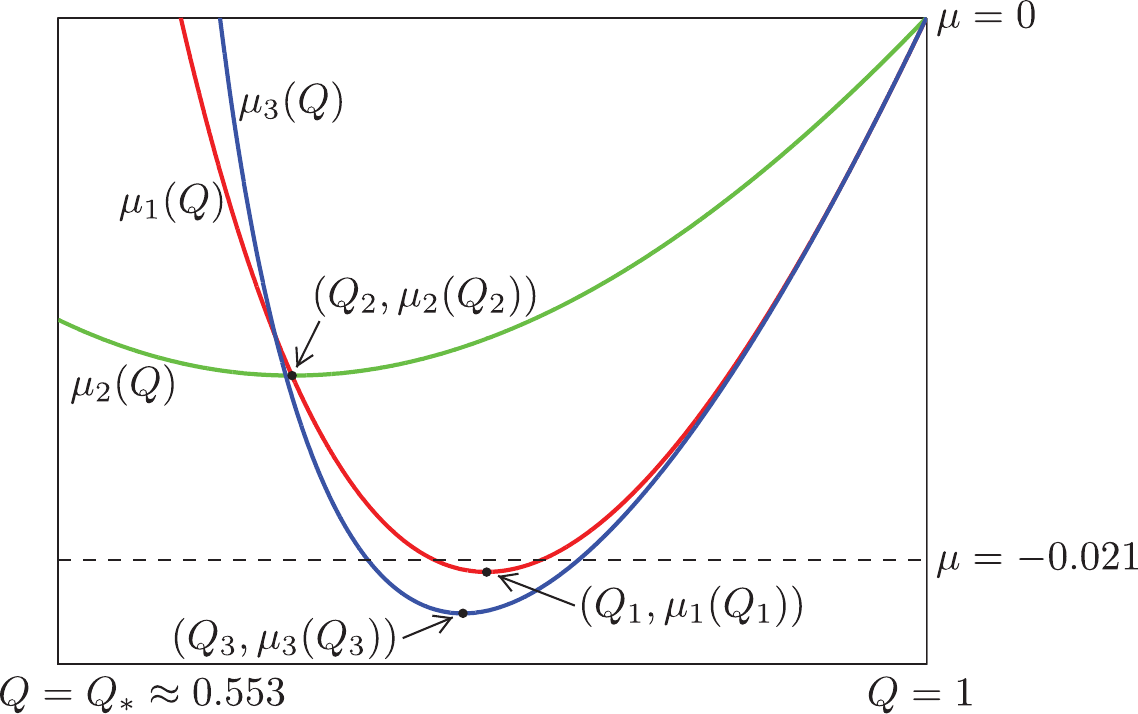}
  \caption{The graphs of $\mu_i(Q)$ with $d=3$, $n = 2.5$, and $c =1.7$. The properties stated in Lemma \ref{LEM mui properties} are true for all $d > 0$, $n \in [2, 3]$, and $c \in [1.55, n)$.}
  \label{FIG MUi}
\end{figure}

\begin{remark}\normalfont
  With the closed-form expressions of \eqref{Fi} for $i=1,2,3$, Lemma
  \ref{LEM mui properties} can be checked by elementary calculations,
  which are omitted. Here, we just mention the following:
  \begin{enumerate}
  \item The condition $c < n$ ensures that
    $Q_1 = (\frac{c}{n})^{\frac{1}{n-1}} < 1$.
  \item Recall that for each $n > 0$, $h_3(Q)$ has a unique zero at
    $Q = Q_* \in (0,1)$, which depends on $n$ only according to
    \eqref{h3}. The condition $c \ge 1.55$ guarantees that for any
    $n \in [2,3]$ and the corresponding $Q_*$,
    $F_3(Q_*,\mu) = g_3(Q_*) + h_3(Q_*) \mu = g_3(Q_*) < 0$. The lower
    bound $1.55$ can be further reduced, but such a lower bound is
    necessary as $g_3(Q_*)$ can become positive for some $n \in [2,3]$
    if $c$ is too small (though still positive).
  \end{enumerate}
\end{remark}

Recall that $Q_r(r) < 0$ for all $r \in (0,\tau)$. By the definition
\eqref{tau} of $\tau$, one of the following three mutually exclusive
cases must occur for a solution $Q(r)$ of the IVP \eqref{IVP}:
\begin{enumerate}
  \renewcommand{\labelenumi}{(\roman{enumi})}
\item $Q_{\tau} < Q_*$. This happens if and only if there exists a
  finite $r_* \in (0, \tau)$ (with either $\tau$ being finite or
  $\tau = \infty$) such that $Q(r_*) = Q_*$ and $Q_r(r) < 0$ for all
  $r \in (0, r_*]$.
\item $\tau < \infty$, and
  $\lim_{r \rightarrow \tau^-} Q(r) = Q_{\tau} \ge Q_*$. In this case,
  as $r \rightarrow \tau^-$, $Q_r(r)$ is bounded due to
  \eqref{bound}. In turn, we have the boundedness of $Q_{rr}(r)$ from
  \eqref{Qrr}, the boundedness of $Q_{rrr}(r)$ from \eqref{ODE}, and
  the boundedness of $Q_{rrrr}(r)$ after further differentiation of
  \eqref{ODE}. It follows that $Q(r) \in C^3([0,\tau])$. Thus, in this
  case, we actually have that $Q(\tau) = Q_{\tau} \ge Q_*$ and
  $Q_r(\tau) = 0$.
\item $\tau = \infty$, and
  $\lim_{r \rightarrow \infty} Q(r) = Q_{\tau} \ge Q_*$.
\end{enumerate}
Define the set $A$ as follows:

\begin{definition} \label{def A}
  $\mu < 0$ is in the set $A$ if and only if case (i) happens for the
  solution $Q(r)$ of the IVP \eqref{IVP} with $Q_{rr}(0) = \mu$.
\end{definition}

\begin{lemma} \label{lem open} $A$ is open.
\end{lemma}

\begin{proof}
  If case (i) happens for $Q_{rr}(0) = \mu$, then there exists a small
  $\epsilon > 0$ such that $Q_{\tau} < Q(r_* + \epsilon) < Q_*$ and
  $Q_r(r) < 0$ for all $r \in (0, r_* + \epsilon]$ with
  $r_*+\epsilon < \tau$. Then by Lemma \ref{lem EUC}, there is a
  $\delta > 0$ such that for each
  $\mu' \in (\mu - \delta, \mu + \delta)$, the solution is below $Q_*$
  at $r = r_*+\epsilon$ and the derivative is negative for all
  $r \in (0, r_* + \epsilon]$. This implies that case (i) also happens
  for all $\mu' \in (\mu - \delta, \mu + \delta)$. Then
  $(\mu - \delta, \mu + \delta) \subseteq A$ by the definition of $A$.
\end{proof}

\begin{lemma} \label{lem subset} $(-\infty, \mu_3(Q_3) ) \subseteq A$.
\end{lemma}

\begin{proof}
  We show by contradiction that cases (ii) and (iii) cannot occur if
  $\mu < \mu_3(Q_3)$.

  Suppose that case (ii) happens. By continuity, evaluating \eqref{DE
    bound} at $r = \tau$ gives $Q_{rr}(\tau) \le F_3(Q_\tau,\mu) < 0$
  as $Q_\tau \ge Q_*$ and $\mu < \mu_3(Q_3)$ (see Figure \ref{FIG
    F3}). On the other hand, since $Q_r(r) < 0$ for all
  $r \in (0, \tau)$ and $Q_r(\tau)=0$, we must have that
  $Q_{rr}(\tau) \ge 0$, which yields a contradiction.

  Suppose that case (iii) happens. Then $Q(r) > Q_{\tau} \ge Q_* > 0$
  for all $r \in (0, \infty)$. It follows that
  $2 F_2(Q(r),\mu) (Q(r))^{-n}$ is bounded for all
  $r \in (0, \infty)$. Then, $Q_r(r)$ is also bounded for all
  $r \in (0, \infty)$ according to \eqref{bound}. From \eqref{DE
    bound}, we have that $Q_{rr}(r) + (d-1) Q_r(r)/r < F_3(Q(r),\mu)$
  for all $r \in (0, \infty)$. Then
  $\limsup_{r \rightarrow \infty} Q_{rr}(r) \le F_3(Q_{\tau},
  \mu)$. As $F_3(Q_{\tau}, \mu) < 0$ for $Q_{\tau} \ge Q_*$ and
  $\mu < \mu_3(Q_3)$, the estimate for
  $\limsup_{r \rightarrow \infty} Q_{rr}(r)$ prohibits the convergence
  of $Q(r)$, producing a contradiction.
\end{proof}

\begin{lemma} \label{lem empty} $(\mu_2(Q_2), 0) \cap A = \emptyset.$
\end{lemma}

\begin{proof}
  Suppose that there exists $\mu_s \in (\mu_2(Q_2), 0) \cap A$. Since
  $\mu_s \in A$, there exists $r_* \in (0, \tau)$ such that
  $Q(r_*) = Q_* > Q_{\tau}$ and $Q_r(r) < 0$ for all $r \in (0,
  r_*]$. By \eqref{bound}, we have that
  $0 < Q_r^2 < 2 F_2(Q,\mu_s) Q^{-n}$ for all $Q \in
  [Q_*,1)$. However, since $\mu_s \in (\mu_2(Q_2), 0)$, there exists
  $Q_s \in (Q_2,1) \subset (Q_*,1)$ such that $\mu_s = \mu_2(Q_s)$. It
  follows that $F_2(Q_s,\mu_s) = F_2(Q_s,\mu_2(Q_s)) = 0$, which gives
  a contradiction.
\end{proof}

Define $\mu_c := \sup A$. By \eqref{mui order} of Lemma \ref{LEM mui
  properties} and Lemmas \ref{lem subset} and \ref{lem empty}, we have
that
\begin{equation*}
\mu_c = \sup A \in [\mu_3(Q_3), \mu_2(Q_2)] < 0.
\end{equation*}
The next lemma proves Proposition \ref{prp Q}.

\begin{lemma} \label{lem existence} For the solution $Q(r)$ of the IVP
  \eqref{IVP} with $\mu = \mu_c$, $Q_r(r) < 0$ for all
  $r \in (0,\infty)$, and $\lim_{r \rightarrow \infty} Q(r) \ge Q_*$.
\end{lemma}

\begin{proof}
  Since $A$ is open by Lemma \ref{lem open},
  $\mu_c = \sup A \not\in A$. Then the only possibilities are cases
  (ii) and (iii). It remains to show that case (ii) does not happen.

  If case (ii) happens, then $Q_{rr}(\tau)$ must be nonnegative since
  $Q_r(r) < 0$ for all $r \in (0, \tau)$ and $Q_r(\tau)=0$. We further
  divide case (ii) into the following three mutually exclusive
  sub-cases:
  \begin{enumerate}
    \renewcommand{\labelenumi}{(ii.\alph{enumi})}
  \item $\tau < \infty$, $Q(\tau) = Q_{\tau} > Q_*$, $Q_r(\tau) = 0$,
    and $Q_{rr}(\tau) > 0$.
  \item $\tau < \infty$, $Q(\tau) = Q_{\tau} > Q_*$, $Q_r(\tau) = 0$,
    and $Q_{rr}(\tau) = 0$.
  \item $\tau < \infty$, $Q(\tau) = Q_{\tau} = Q_*$, and
    $Q_r(\tau) = 0$.
  \end{enumerate}

  If (ii.a) happens, we can further extend the solution. Specifically,
  since $Q_r(\tau) = 0$ and $Q_{rr}(\tau) > 0$, there is a small
  $\epsilon > 0$ such that $Q(r) \ge Q_{\tau} > Q_*$ for all
  $r \in [0, \tau+\epsilon]$, $Q_r(r) < 0$ for all
  $r \in (0, \tau-\epsilon]$, and $Q_r(\tau+\epsilon) > 0$. Then by
  Lemma \ref{lem EUC}, there is a $\delta > 0$ such that for each
  $\mu \in (\mu_c - \delta, \mu_c + \delta)$, the solution is above
  $Q_*$ for all $r \in [0, \tau+\epsilon]$ with a local minimum in the
  interval $[\tau-\epsilon, \tau+\epsilon]$. This implies that
  $(\mu_c - \delta, \mu_c + \delta) \cap A = \emptyset$ by the
  definition of $A$. This cannot be true since $\mu_c = \sup A$.

  Case (ii.b) is impossible since the ODE \eqref{ODE} with the initial
  condition $Q(\tau) = Q_{\tau}$, $Q_r(\tau) = 0$, and
  $Q_{rr}(\tau) = 0$ has a unique solution $Q(r) \equiv Q_{\tau}$.

  If (ii.c) happens, then evaluating \eqref{DE bound} at $r = \tau$
  gives $Q_{rr}(\tau) \le F_3(Q_*,\mu_c) = g_3(Q_*)< 0$. This is
  impossible since $Q_{rr}(\tau)$ must be nonnegative when case (ii)
  happens.
\end{proof}

\subsection{Proof of Proposition \ref{prp suff cond} and Example for Exponential Convergence} \label{prf prp exp conv}

First we prove:
\begin{lemma} \label{lem Qr} If $Q_r(r) < 0$ for all
  $r \in (0,\infty)$ and
  $\lim_{r \rightarrow \infty} Q(r) = Q_{\tau} > 0$, then
  $\lim_{r\rightarrow\infty}Q_r(r) = 0$ and
  $\lim_{r\rightarrow\infty}Q_{rr}(r) = 0$.
\end{lemma}

\begin{proof}
  By the hypothesis of the lemma, we have \eqref{Qrr}, \eqref{bound},
  and \eqref{DE bound} for all $r \in (0, \infty)$ and
  $Q = Q(r) \in (Q_{\tau}, 1)$. In the limit $r \rightarrow \infty$
  (and correspondingly $Q \rightarrow Q_{\tau}^+$), inequality
  \eqref{bound} implies that $|Q_r|$ is bounded. Then by \eqref{DE
    bound}, there is some $M>0$ such that $Q_{rr} < M$ for all
  $r \in (0, \infty)$ and correspondingly for all
  $Q = Q(r) \in (Q_{\tau}, 1)$. Suppose $Q_r$ does not converge to $0$
  as $r \rightarrow \infty$. Recall that $Q_r(r) < 0$ for all
  $r \in (0,\infty)$. Then we can find an $\epsilon > 0$ and a
  sequence $\{ r_i>0 \}_{i=1}^{\infty}$ with
  $r_i + \frac{\epsilon}{M} < r_{i+1}$ such that
  $Q_r(r_i) < -\epsilon$ for all $i$. It follows that
  \begin{equation*}
  \int_{r_1}^{\infty} Q_r dr < \sum_{i=1}^{\infty}
  \int_{r_i}^{r_i+\frac{\epsilon}{M}} Q_r dr < \sum_{i=1}^{\infty}
  \int_{r_i}^{r_i+\frac{\epsilon}{M}} (-\epsilon+M(r-r_i)) dr =
  \sum_{i=1}^{\infty} \frac{-\epsilon^2}{2M} \longrightarrow -\infty.
  \end{equation*}
  This is impossible as $Q(r)$ needs to converge to
  $Q_{\tau} > 0$. Next, taking the limit $r \rightarrow \infty$ in
  \eqref{Qrr}, we have that $Q_{rr}$ converges to a
  constant. Furthermore, since $Q_r$ converges, $Q_{rr}$ must converge
  to $0$ as $r \rightarrow \infty$.
\end{proof}

With this, we can now prove Proposition \ref{prp suff cond}.

\begin{proof} (Proposition \ref{prp suff cond}) Taking the limit
  $r \rightarrow \infty$ on both sides of \eqref{Qr2}, we obtain that
  the right-hand side of \eqref{Qr2} is zero when the upper limits of
  the two integrals are both $Q_{\tau}$. Since
  $\int_1^Q = \int_1^{Q_{\tau}} + \int_{Q_{\tau}}^Q$, we have that
  \begin{align}
    \frac{1}{2} Q^n Q_r^2 
    ={}& \int_{Q_{\tau}}^Q F_1(q,\mu) q^n dq - \int_{Q_{\tau}}^Q \left( \frac{n}{2} \int_1^q \frac{Q_r^2}{\hat{q}^2} d\hat{q} + (d-1) \frac{Q_r}{r} \right) q^n dq \notag \\
    >{}& \int_{Q_{\tau}}^Q F_1(q,\mu) q^n dq - \int_{Q_{\tau}}^Q \left( \frac{n}{2} \int_1^q \frac{Q_r^2}{\hat{q}^2} d\hat{q} \right) q^n dq , \label{LB on Qr}
  \end{align}
  where the inequality follows from the hypothesis that $Q_r(r) < 0$
  and $Q(r) > Q_{\tau}$ for all $r \in (0,\infty)$. Differentiating
  \eqref{LB on Qr} with respect to $Q$ yields that
  \begin{equation*}
  \left( F_1(Q,\mu) - \frac{n}{2} \int_1^Q \frac{Q_r^2}{\hat{q}^2} d\hat{q}
  \right) Q^n ,
  \end{equation*}
  which is zero in the limit
  $Q \rightarrow Q_{\tau}^+$. This can be shown by taking the limit
  $r \rightarrow \infty$ on both sides of \eqref{Qrr}. Differentiating
  the above expression with respect to $Q$ once more gives
  \begin{equation*}
  \left( \frac{d}{dQ}F_1(Q,\mu) - \frac{n Q_r^2}{2 Q^2} \right) Q^n +
  \left( F_1(Q,\mu) - \frac{n}{2} \int_1^Q \frac{Q_r^2}{\hat{q}^2}
    d\hat{q} \right) nQ^{n-1} .
  \end{equation*}
  In the limit
  $Q \rightarrow Q_{\tau}^+$, the second part of the above expression
  is zero while by Lemma \ref{lem Qr} the first part converges to
  \begin{equation*}
  L:=\left.\frac{d}{dQ}F_1(Q,\mu) \right|_{Q=Q_{\tau}} =
  \frac{1}{Q_{\tau}^n} - \frac{n}{c Q_{\tau}} ,
  \end{equation*}
  which is positive
  for $Q_{\tau} < (\frac{c}{n})^{\frac{1}{n-1}}$. It follows that
  $\frac{1}{2} Q^n Q_r^2 > \frac{1}{2} L (Q - Q_{\tau})^2 > 0$ for
  $Q(r)$ sufficiently close to $Q_{\tau}$. This implies the
  exponential convergence of $Q(r)$ towards $Q_{\tau}$.
\end{proof}

Below we demonstrate a way to check the conditions of Proposition \ref{prp suff cond} for specific values of $d$, $n$, and $c$ using a combination of analysis and numerical integration. Although we have selected some particular
choices for $d$, $n$, and $c$ here, we note that this strategy is deployable for other parameter values as well. Furthermore, this strategy can be made into a computer-assisted proof if the numerical computation is done using rigorous numerics.

\begin{example} \label{prp exp conv} For $d=3$, $n=2.5$, and $c=1.7$, numerical integration of (\ref{IVP}) provides evidence that the conditions of Proposition \ref{prp suff cond} hold. Thus it is believed that the monotonic solution $Q(r)$ to the IVP \eqref{IVP} with $\mu=\mu_c$ converges exponentially to $Q_{\tau}$ as described in Proposition \ref{prp suff cond}.
\end{example}

\begin{proof}[Explanation of Example \ref{prp exp conv}]
Substituting
  $\mu = \mu_c$ and taking the limit $Q(r) \rightarrow Q_{\tau}^+$ in
  \eqref{Qrr}, with Lemma \ref{lem Qr} we obtain that
  \begin{equation*}
  F_1(Q_{\tau},\mu_c) = \frac{n}{2} \int_1^{Q_{\tau}}
  \frac{Q_r^2}{q^2} dq < 0.
  \end{equation*}
  In addition, from \eqref{DE bound}
  and Lemma \ref{lem Qr} we have that
  \begin{equation*}
  F_3(Q_{\tau},\mu_c) > 0.
  \end{equation*}
  Thus, the point $(Q_{\tau},\mu_c)$ must lie in the region below the graph of $\mu_1(Q)$, where $F_1(Q,\mu)<0$, and above the graph of $\mu_3(Q)$, where $F_3(Q,\mu)>0$ (see Figure \ref{FIG Fi} and Figure \ref{FIG MUi}). Note that item (4) of Lemma \ref{LEM mui properties} guarantees that such a region is nonempty.  

  Recall that $\mu_1(Q)$ has a global minimum in the interval
  $(Q_*, 1)$ at $Q= Q_1 = (\frac{c}{n})^{\frac{1}{n-1}}$. For $d = 3$, $n = 2.5$, and $c = 1.7$, the minimum $\mu_1(Q_1) \approx -0.02146 < -0.021$ (see Figure \ref{FIG MUi}). If for $\mu = -0.021$ there exists a finite $r_*$ such that the solution $Q(r)$ to the IVP \eqref{IVP} satisfies that $Q(r_*) = Q_*$ and $Q_r(r) < 0$ for all $r \in (0, r_*]$, then $\mu = -0.021 \in A$ by Definition \ref{def A}. It follows that $\mu_c = \sup A > -0.021 > \mu_1(Q_1)$. Since the point $(Q_{\tau},\mu_c)$ now must also be above the line $\mu = -0.021$ in addition to being bounded between the graphs of $\mu_1(Q)$ and $\mu_3(Q)$, it can only be in either the region with $Q_{\tau} < Q_1 = (\frac{c}{n})^{\frac{1}{n-1}}$ or the region with $Q_{\tau} > Q_1 = (\frac{c}{n})^{\frac{1}{n-1}}$ (see Figure \ref{FIG MUi}). The latter is impossible since the linearization of \eqref{ODE} at such a $Q_{\tau}$ has oscillatory dynamics that forbid a solution from converging to $Q_{\tau}$ monotonically. Then by Proposition \ref{prp suff cond}, exponential convergence ensues.
  
  Finally, we present numerical evidence that for $d = 3$, $n = 2.5$, and $c = 1.7$, the solution to the IVP \eqref{IVP} with $\mu = -0.021$ decreases monotonically and drops below $Q_*$ in finite $r$. We numerically integrate the ODE \eqref{ODE}. However, to circumvent the singularity at $r=0$, we first seek a series solution to the IVP \eqref{IVP} in the form of $Q(r) = 1 + \frac{1}{2}\mu r^2 + \sum_{k=3}^{\infty} a_k r^k$. For $\mu = -0.021$, we obtain that 
  \begin{align*}
  Q(r) \approx 1 - 0.0105 r^2 + 0.0002194963 r^4 - 0.0000027819 r^6 + O(r^8) .
  \end{align*}
  Evaluating the above polynomial approximation and its first and second derivatives at $r=0.01$ gives that $Q(0.01) \approx 0.9999989500$, $Q_r(0.01) \approx -0.0002099991$, and $Q_{rr}(0.01) \approx -0.0209997366$. Then we use these values as the initial conditions at $r=0.01$ and integrate \eqref{ODE} using the MATLAB ode45 function with both the relative error tolerance and the absolute error tolerance set to $10^{-10}$. The plots of the numerical solutions for $Q(r)$ and $Q_r(r)$ are shown in Figure \ref{FIG NumSol}. Notice that for $r_* \approx 16.969$, $Q(r_*) = Q_*$ and $Q_r(r) < 0$ for all $r \in (0, r_*]$.
\end{proof}

\begin{figure}[h]
  \centering 
    \subfigure[$Q$ vs $r$.]{\label{FIG Q}\includegraphics[scale=0.277]{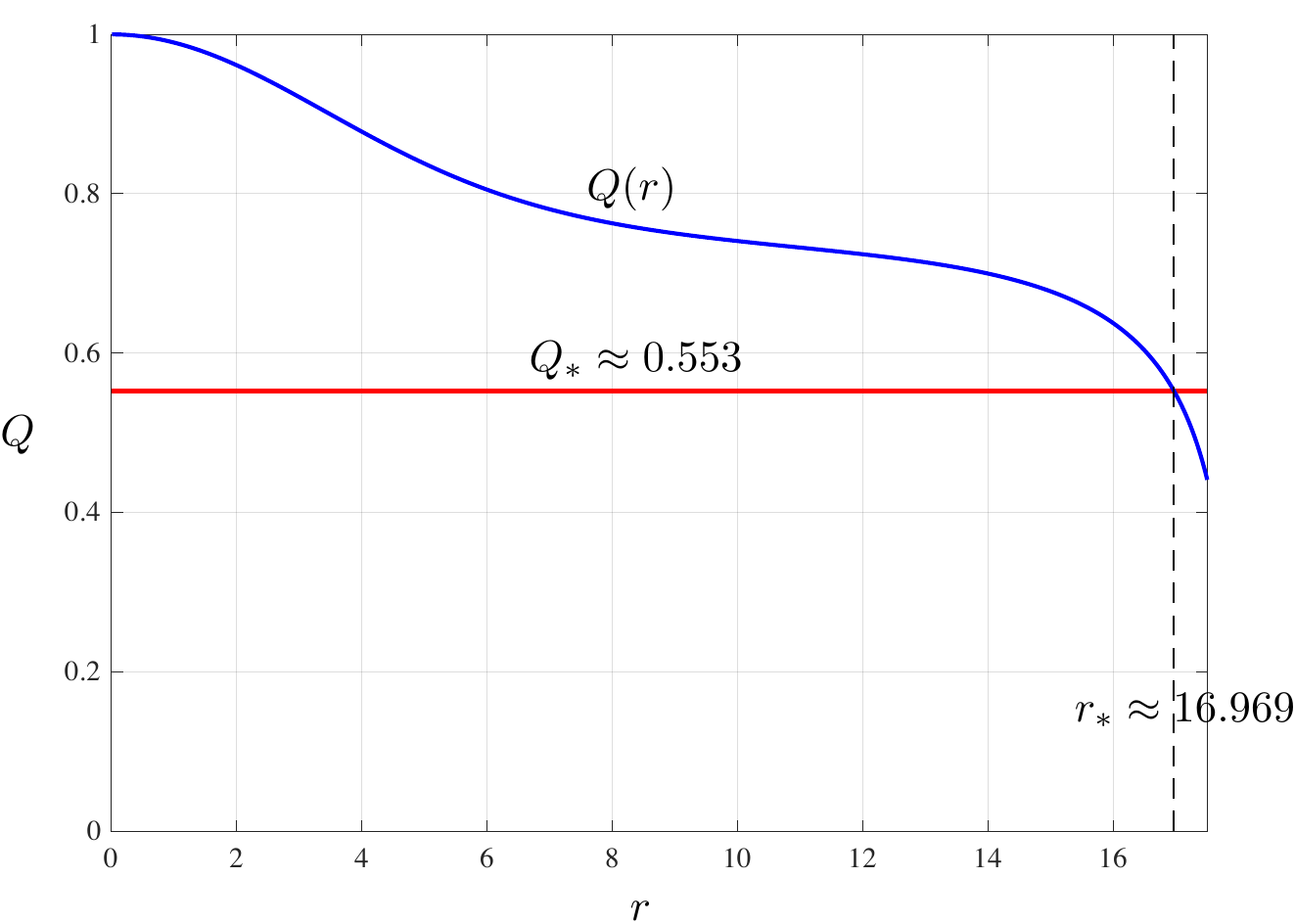}} \hspace{0.01in}
    \subfigure[$Q_r$ vs $r$.]{\label{FIG Qr}\includegraphics[scale=0.277]{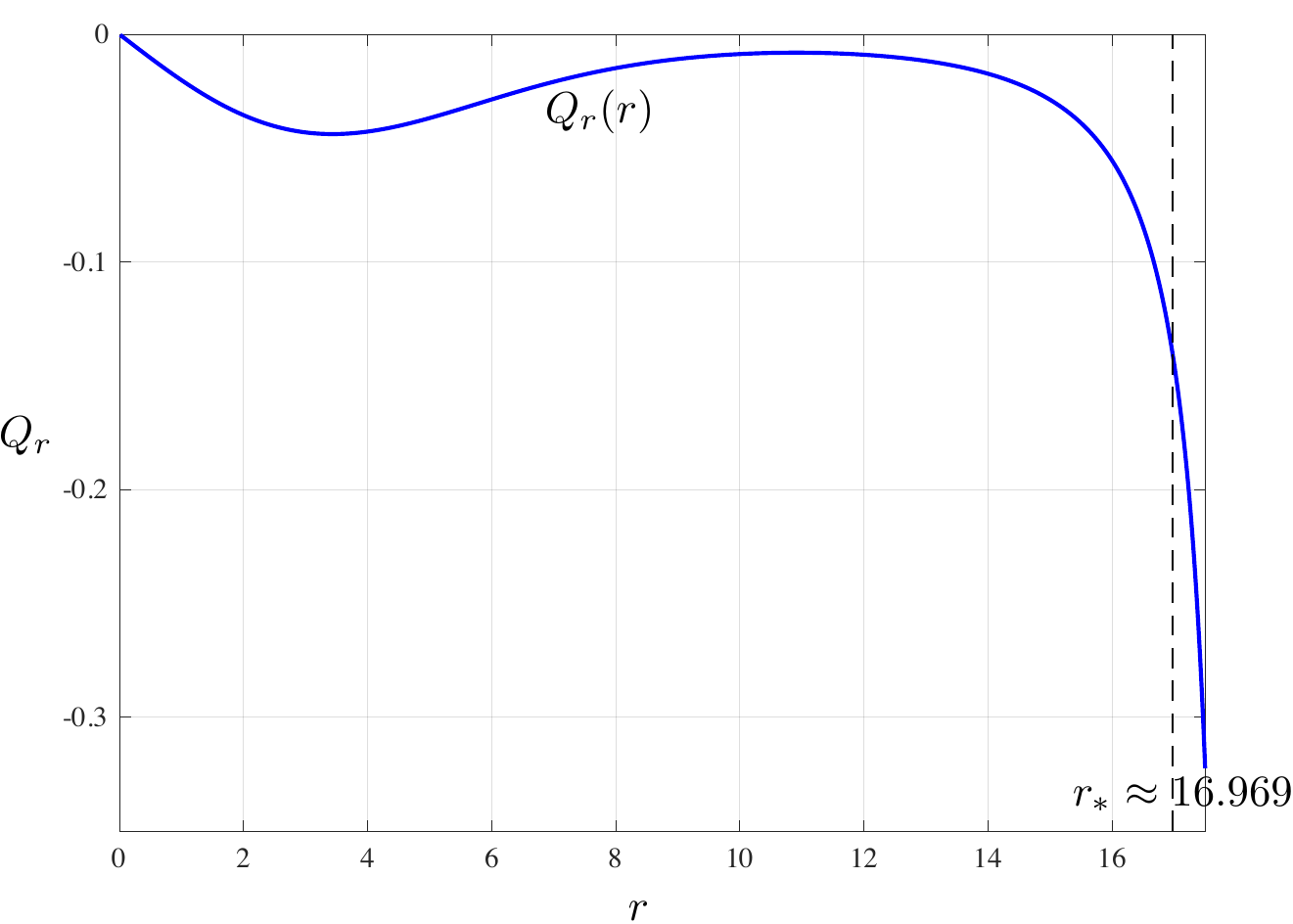}}
  \caption{The plots of $Q(r)$ and $Q_r(r)$, obtained by numerically integrating the ODE \eqref{ODE} with the initial conditions $Q(0.01) = 0.9999989500$, $Q_r(0.01) = -0.0002099991$, and $Q_{rr}(0.01) = -0.0209997366$. The parameter values are $d = 3$, $n = 2.5$, and $c = 1.7$. $Q_* \approx 0.553$ is the unique zero in $(0, 1)$ of the function $h_3(Q)$ given by \eqref{h3}. Note that $Q(r)$ crosses $Q_*$ at $r_* \approx 16.969$ and that $Q_r(r)$ is negative for all $r \in (0, r_*]$.}
  \label{FIG NumSol}
\end{figure}

\section{Discussion}
\label{s:discussion}

In this work we have established the well-posedness of solutions to
the time dependent problem \eqref{magma}, and demonstrated the
existence of monotonically decaying traveling wave solutions.  The
major advancement here was to obtain such results in dimensions two
and higher.  Here, we remark on several aspects of our results.

First, our results were obtained in the ``constant bulk viscosity'' case
and with integer nonlinearity $n$ for the permeability.
This was largely to simplify presentation, and we believe our results
could be extended via the same methods to the more general case,
\begin{equation*}
  \phi_t + \partial_{x_d} \left( \phi^n \right) - \nabla \cdot \left(\phi^n\nabla (\phi^{-m}\phi_t)\right) = 0,
\end{equation*}
allowing for non-integer exponents.  Typical values of $m$ are $m\in[0,1].$

Next, with respect to the time-dependent problem, we made no mention
of global-in-time results.  Global-in-time results were obtained in
\cite{Simpson:2007kx}  in $d=1$ for certain choices of the nonlinearity by
making use of conservation laws of the form, in $d=1$,
\begin{equation}
\label{e:conlaw}
\int \frac{1}{2} |\phi^{-m}\partial_x\phi|^2 + \frac{\phi^{2-n-m} -1 + (n+m-2)(\phi-1)}{(n+m-1)(n+m-2)}
\end{equation}
provided $n+m\neq 1,2$.  An inspection of this shows that for
$\phi-1\in H^1$, it is nonnegative and convex about the $\phi =1$
solution.  Other expressions hold in the cases $n+m =1$
and $n+m=2$.  For certain values of $n$ and $m$, \eqref{e:conlaw}
provides {\it a priori} an upper bound on the $H^1$ norm and a lower
bound on $\phi$, preventing it from going to zero.  This lower bound
makes use of the Sobolev embedding of $H^1$ into $L^\infty$, which
does not hold in higher dimensions.  Since our analysis requires
pointwise control of a lower bound, any analog would require {\it a
  priori} estimates on the higher index Sobolev spaces.
Unfortunately, for general nonlinearities, no
higher order conservation laws are anticipated (see
\cite{Harris:1999if} for some exceptions).

With regard to our main result on the solitary waves, Theorem \ref{t:main2}, 
we note that our result is somewhat different from what might be expected
by the computational geophysics community.  In, for instance, 
\cite{Simpson:2011fx, Scott:1986kf,scott1984magma,Barcilon:1986wd,Barcilon:1989ve,Wiggins:2012iv}, a value of 
$\bar{c}>n$ is specified, and then a unimodal profile is obtained numerically that is 
observed to decay exponentially fast towards one.  In those works, $\bar{c}$ is the 
only free parameter and the amplitude is unknown.  Here, due to the rescaling \eqref{res}, 
for each $1.55 < c < n$ we obtain a unimodal profile whose amplitude is fixed at $1$, 
but whose limiting value $Q_{\tau}$ is unknown. Note that if we choose $\bar{q}_0 = Q_{\tau}^{-1}$ 
in \eqref{res}, then $\bar{c} = Q_{\tau}^{1-n} c$. In this case, the profile $\bar{Q}$ 
decays to $1$ at infinity, and the condition $Q_{\tau} < (\frac{c}{n})^{\frac{1}{n-1}}$ for
exponential decay (Proposition \ref{prp suff cond}) is exactly equivalent to $\bar{c}>n$. On 
the other hand, since we do not know the dependence between $Q_{\tau}$ and $c$, our result 
is unable to guarantee that $\bar{c} = Q_{\tau}^{1-n} c$ can be matched to a pre-specified value. 

While we have not succeeded at proving exponential decay, we have provided a criterion on the solitary wave profile, which, if satisfied, ensures that such a profile decays exponentially to $Q_{\tau}$; this criterion states that if $Q_{\tau} < (\frac{c}{n})^{\frac{1}{n-1}}$, then the exponential decay does in fact occur. We have also demonstrated that for some specific parameter values, a mix of analysis and numerical evidence suggests that this criterion is satisfied.

\section{Acknowledgements}
We gratefully acknowledge the National Science Foundation, which has
generously supported this work through grants DMS-1515849 [DMA],
DMS-1409018 [GRS] and DMS-1511488 [JDW].

\bibliographystyle{plain}

\bibliography{magma_refs_rev}

\end{document}